\documentclass[a4paper,12pt]{amsart}

\usepackage{a4wide,hyperref,amsmath,amssymb,enumitem, stmaryrd,xcolor,comment}
\usepackage{tikz}

\usetikzlibrary{decorations.pathreplacing}
\usetikzlibrary{shapes}
\tikzset{
			inner sep=1pt,semithick,
			vertex/.style={circle,draw,fill,minimum size= 1pt},
			vertexb/.style={rectangle,draw,fill, minimum size = 3pt},
			vertexc/.style={circle,draw,fill=white,minimum size= 1pt},			
			thickedge/.style={line width=0.73pt},
			rededge/.style={line width=0.73pt, red},
			blueedge/.style={line width=0.73pt, blue},						
			font=\tiny
}

\newcommand{\EE}{\mathbb{E}}

\newcommand{\LL}{\mathbb{L}}

\newcommand{\id}{\operatorname{id}}
\newcommand{\gap}{\operatorname{gap}}
\newcommand{\mast}{\operatorname{mast}}
\newcommand{\mastcat}{\operatorname{mastcat}}
\newif\ifdetails
\detailstrue
\newcommand{\DETAIL}[1]%
{\ifdetails\par\fbox{\begin{minipage}{0.9\linewidth}\textit{Detail:}
      #1\end{minipage}}\par\fi}
\newcommand{\TODO}[1]%
{\ifdetails\par\fbox{\begin{minipage}{0.9\linewidth}\textbf{TODO:}
      #1\end{minipage}}\par\fi}

\newtheorem{lemma}{Lemma}
\newtheorem{prop}[lemma]{Proposition}
\newtheorem{theorem}[lemma]{Theorem}
\newtheorem{corollary}[lemma]{Corollary}

\newtheorem{remark}{Remark}

\newtheorem{question}{Question}

\newtheorem{claim}{Claim}
\newtheorem{observation}{Observation}

\newcommand{\old}[1]{{}}

\usepackage{lineno}

\makeatletter
\DeclareRobustCommand{\cev}[1]{%
  {\mathpalette\do@cev{#1}}%
}
\newcommand{\do@cev}[2]{%
  \vbox{\offinterlineskip
    \sbox\z@{$\m@th#1 x$}%
    \ialign{##\cr
      \hidewidth\reflectbox{$\m@th#1\vec{}\mkern4mu$}\hidewidth\cr
      \noalign{\kern-\ht\z@}
      $\m@th#1#2$\cr
    }%
  }%
}
\makeatother

\title{Coconvex characters on collections of phylogenetic trees}

\author{Eva Czabarka}
\address{University of South Carolina}

\author{Steven Kelk}
\address{Maastricht University}

\author{Vincent Moulton}
\address{University of East Anglia}

\author{L\'aszl\'o A. Sz\'ekely} 
\address{University of South Carolina}

\date{\today}

\begin{document}

\begin{abstract}
    In phylogenetics, a key problem is to construct evolutionary trees
    from collections of characters where, for a set $X$ of species, 
    a character is simply a function from $X$ onto a set of states. 
    In this context, a key concept is convexity, where a character
    is convex on a tree with leaf set $X$ if the collection of subtrees spanned
    by the leaves of the tree that have the same state are pairwise disjoint.
    Although
    collections of convex characters on a single
    tree have been extensively studied over the past few decades, very little is known
    about {\em coconvex characters}, that is, characters that are simultaneously
    convex on a collection of trees. As a starting point to better 
    understand coconvexity, in this paper 
    we prove a number of extremal results for the following question:
    {\em What is the minimal number of coconvex characters  
    on a collection of $n$-leaved trees taken over all collections of size $t \ge 2$, also  
    if we restrict to coconvex characters which map to $k$ states?}
    As an application of coconvexity, we introduce a new one-parameter 
    family of tree metrics, which range between the coarse Robinson-Foulds distance 
    and the much finer quartet distance.
    We show that bounds on the quantities in the above question 
    translate into bounds for the diameter of the tree space for the new distances.
    Our results open up several new interesting directions and questions which 
    have potential applications to, for example, tree spaces and phylogenomics.
\end{abstract}


\maketitle

\newpage 


\section{Introduction}

In phylogenetics a key problem is to infer evolutionary trees based on data
arising from contemporary species \cite{SS03,steel2016phylogeny}.
Mathematically speaking, given a collection of species $X$, the aim is to construct a {\em phylogenetic tree}, that is, 
a binary tree with leaf set $X$
whose branching patterns indicate how the species evolved over time.
Often evolutionary data comes in the form of {\em characters}, that is,
maps from $X$ into some set of states. 
For example, if $X$ is a collection of animals a character could 
be the map that assigns each animal in $X$ to the state of either having wings or no wings.
Each character induces a partition on $X$ where two elements in $X$ 
are in the same class if and only if the character maps them to the same state (see e.g. Fig.~\ref{fig:coconvex}).
In this paper we shall consider two characters to be equivalent if they induce the
same partition and thus, in particular, we consider characters as being 
partitions of $X$ and the states of the character as being the classes 
in the  associated partition.

In the context of constructing phylogenetic trees from characters a key concept is convexity.
Formally, a character is {\em convex on a phylogenetic tree $T$} if the subtrees of $T$
spanned by the leaves in each of its partition classes are pairwise disjoint in $T$ 
(see e.g. Fig.~\ref{fig:coconvex}). Convexity is motivated 
by the concept of {\em homoplasy-free} evolution in which a biological state (e.g. wings versus no wings)
evolve only once (cf. \cite[p.67]{SS03}). Characters that contain at most one 
class of size at least 2 are {\it trivial}, as they are convex on every phylogenetic tree.
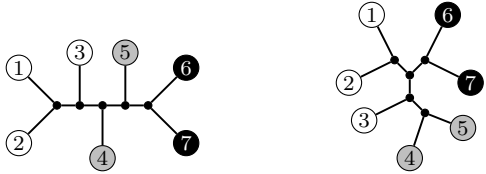
\begin{figure}[h]
\centering
\begin{tikzpicture}
\node[vertex,fill=white] (l1) at (-.5,.5) {$1$};
\node[vertex,fill=white] (l2) at (-.5,-.5) {$2$};
\node[vertex] (i1) at (0,0) {};
\node[vertex] (i2) at (.3,0) {};
\node[vertex, fill=white] (l3) at (.3,.7) {$3$};
\node[vertex] (i3) at (.6,0) {};
\node[vertex, fill=gray!50] (l4) at (.6,-.7) {$4$};
\node[vertex] (i4) at (.9,0) {};
\node[vertex, fill=gray!50] (l5) at (.9,.7) {$5$};
\node[vertex,fill=black,text=white] (l6) at (1.7,.5) {$6$};
\node[vertex,fill=black,text=white] (l7) at (1.7,-.5) {$7$};
\node[vertex] (i5) at (1.2,0) {};
\draw[thickedge] (l1)--(i1)--(l2);
\draw[thickedge] (i1)--(i2)--(l3);
\draw[thickedge] (i2)--(i3)--(l4);
\draw[thickedge] (i3)--(i4)--(l5);
\draw[thickedge] (i4)--(i5)--(l6);
\draw[thickedge] (i5)--(l7);
\end{tikzpicture}
\qquad\qquad
\begin{tikzpicture}
\node[vertex,fill=white] (l1) at (-.3,.6) {$1$};
\node[vertex,fill=white] (l2) at (-.6,-.3) {$2$};
\node[vertex] (i1) at (0,0) {};
\node[vertex] (i2) at (.2,-.2) {};
\node[vertex, fill=white] (l3) at (-.4,-.8) {$3$};
\node[vertex] (i3) at (.4,0) {};
\node[vertex, fill=gray!50] (l4) at (.2,-1.3) {$4$};
\node[vertex] (i4) at (.2,-.5) {};
\node[vertex, fill=gray!50] (l5) at (.9,-.9) {$5$};
\node[vertex,fill=black,text=white] (l6) at (.7,.6) {$6$};
\node[vertex,fill=black,text=white] (l7) at (1,-.3) {$7$};
\node[vertex] (i5) at (.4,-.7) {};
\draw[thickedge] (l1)--(i1)--(l2);
\draw[thickedge] (i1)--(i2)--(i3);
\draw[thickedge] (l6)--(i3)--(l7);
\draw[thickedge] (i2)--(i4)--(l3);
\draw[thickedge] (i4)--(i5)--(l5);
\draw[thickedge] (i5)--(l4);
\end{tikzpicture}
\caption{Two phylogenetic trees with leaf set $\{1,\dots,7\}$. 
The character $f:X \to \{A,G,T\}$
defined by $f(1)=f(2)=f(3)=A$, $f(4)=f(5)=G$, and $f(6)=f(7)=T$ 
induces the partition $\{\{1,2,3\},\{4,5\},\{6,7\}\}$. On the picture, white, gray and black colors on the leaves correspond to $A$, $G$, and $T$ respectively.
 In particular, $f$ is convex on
both  of the trees, and thus it is coconvex on the trees. In contrast, 
any character $g$ on $X$ which induces the partition 
$\{\{1,2,3,4\},\{5,6,7\}\}$ is not coconvex on these two trees, since $g$ is convex on left tree 
but not on the right one. }\label{fig:coconvex}
\end{figure}

In the past few decades there has been a great deal of work on understanding 
collections of convex characters on a single phylogenetic tree 
(see e.g. \cite[Chapter 4]{SS03},
\cite[Chapter 5]{steel2016phylogeny}).
Even so, to our best knowledge, very little is known about the behavior of collections of
characters that are simultaneously convex on a set of phylogenetic trees.
We call such a character {\em coconvex} (see e.g. Fig.~\ref{fig:coconvex}).   

The concept of coconvexity has recently 
surfaced in the algorithmic phylogenetics 
literature \cite{kelk2022sharp,kelk2024sidma,convex}. 
In particular, it has been observed that \emph{agreement forests}, combinatorial 
objects used to summarize the dissimilarity of a set of phylogenetic trees,  
project down onto characters that are coconvex. 
Counting (and enumerating) coconvex characters can thus, for example, provide an 
algorithmic bridge into counting and enumerating more complex phylogenetic structures.
Coconvexity could also potentially useful in present day phylogenomic studies 
which often involve processing collections of gene trees in order 
to construct phylogenetic trees (see e.g. \cite{Lozano-Fernandez2022}).

As a starting point to better understand coconvexity, in this paper
we prove a number of extremal results for the following question.

\begin{question} \label{question}
What is $s_n^{(t)}$,
the minimum number of coconvex characters 
on a collection of $n$-leaved trees taken over all collections of size $t \ge 2$? 
What is $s_{n,k}^{(t)}$,
the minimum number of coconvex characters with $k$ classes
on a collection of $n$-leaved trees taken over all collections of size $t \ge 2$? 
In particular, what is the smallest $k$, for which $s_{n,k}^{(t)}$ is more than the number of trivial characters?
\end{question}

Note that any trivial  character on $X$ 
is coconvex on any collection of $n$-leaved trees. It follows that 
$s_n^{(t)} \ge 2^n-n$ (see also Lemma~\ref{lm:monotone} below).
Moreover, it is known that there are at most $F_{2n-1}=O(2.619^n)$ convex characters on any 
$n$-leaved tree \cite{convex,SteelFibonacci}, where $F_m$ denotes the 
$m^{th}$ Fibonnaci number, and so this is also an upper bound for $s_n^{(t)}$.

We now briefly summarise the content of the rest of this paper.
Most of our results will focus on the special case of Question~\ref{question} where
all trees in the collection are {\em caterpillars}, that is, trees for which 
the removal of all leaves from the tree results in a path (see e.g., the tree in the left of Fig.~\ref{fig:coconvex}). In this
situation, we denote the quantities analogous to $s_n^{(t)}$ and $s_{n,k}^{(t)}$
by $c_n^{(t)}$ and $c_{n,k}^{(t)}$, respectively. 

After presenting some preliminaries, we then concentrate 
on the case $t=2$. In Section~\ref{sec:inequalities}, we begin by 
presenting some basic relationships between
the quantities $s_n^{(2)}$, $s_{n,k}^{(2)}$, $c_n^{(2)}$ and $c_{n,k}^{(2)}$.
Then, in Sections~\ref{sec:lower-bounds} to \ref{sec:upper}, we focus on finding bounds for 
pairs of caterpillar trees. More specifically, 
in Theorem~\ref{symmet} and Corollary~\ref{lowerbound_c_n}
we give lower bounds for $c_{n,k}^{(2)}$ and $c_n^{(2)}$, respectively,
as well as precise values for $c_{n,k}^{(2)}$, which come from trivial characters,  when $k$ is at most $\lceil \frac{n}{3}\rceil$ in Theorem~\ref{megegyezik}. Theorem~\ref{fo}
determines up to constant multiplicative factor the expected number of non-trivial coconvex characters
of two randomly and uniformly selected caterpillar trees.
This yields in turn  an upper bound for $c_n^{(2)}$. 

In Section~\ref{sec:more}, we consider more than two trees, and use
some results from the theory of maximum agreement subtrees to give a lower bound
for $s_{n,k}^{(t)}$ and $c_{n,k}^{(t)}$.  We also provide some precise values for $c_{n,k}^{(t)}$, when it boils down to the trivial characters with exactly $k$ classes.

In Section~\ref{sec:metrics} we present an application of coconvexity.
Since many phylogenetic reconstruction methods optimize an objective 
function over the collection of phylogenetic trees on the same leaf set,
the structure of such tree spaces have been studied extensively, 
see e.g., \cite{St.Johns,Steel1993}. 
In particular, metrics for comparing phylogenetic
trees with $n$ leaves, or {\em tree metrics}, are important in applications
such as the reconstruction of evolutionary trees where they 
help guide algorithmic searches for optimal trees \cite{Steel1993}. 
Since different metrics can reveal alternative insights 
to the structure of tree spaces \cite{St.Johns}, it is topical 
in the phylogenetic literature to look for new metrics;  see
e.g. \cite{bogdanowicz,briand,cardona,chauve2025vector}. 

In Section~\ref{sec:metrics} we
introduce a new family of metrics $d_k$, for $2 \le k \le n-2$,
on the set of $n$-leaved phylogenetic trees.
Our new distances $d_k$ provide a 1-parameter family of metrics
that interpolate between the extremes of the well-known Robinson-Foulds \cite{Robinson1981} and 
quartet distances \cite{bandeltdress,Steel1993}, which are identical to our $d_2$ and $d_{n-2}$, respectively.  
As we shall see,  bounds on the
quantity $s_{n,k}$ translate into  bounds for the diameter of the 
metric $d_k$. While the diameter of the Robinson-Foulds metric is 
easily seen to be $2n-6$ \cite{Steel1993}, the diameter of the quartet distance is a
 40 year-old open problem. A conjecture of   Bandelt and Dress \cite[p.338]{bandeltdress}  asserts that the diameter is at most
$(2/3+o(1))\binom{n}{4}$, achieved by caterpillar trees.  
Noga Alon, Humberto Naves, and Benny Sudakov \cite{alon} 
proved that the diameter is at most
$(0.69+o(1))\binom{n}{4}$, and verified that the conjecture is true
for the {\em diameter of caterpillar trees}, instead of all phylogenetic trees. 


We conclude the paper in Section~\ref{sec:discuss} by stating some open problems.

\section{Preliminaries}\label{sec:preliminaries}

In what follows we shall use the language of partitions rather than characters as
this is more convenient for our arguments.
Let $\mathfrak{P}_n$ denote the set of partitions of the set $[n]=\{1,2,...,n\}$ into non-empty classes. 
In the future, partition classes are
non-empty by definition. Let $\mathfrak{Q}_{n,t}$ denote the set of partitions of the set $[n]$ with exactly $t$ 
classes of size at least two (note $t\ge 0$), and
$\mathfrak{Q}_{n}=\bigcup_{t\ge 2}\mathfrak{Q}_{n,t}=\mathfrak{P}_n\setminus(\mathfrak{Q}_{n,0}\cup\mathfrak{Q}_{n,1})$ be 
the set of partitions with at least two classes of size at least two. 

For $\mathcal{P}\in \mathfrak{P}_n$,
we let $| \mathcal{P}|$ denote the number of classes in $\mathcal{P}$. Among the partitions of the set $[n]$
we denote by $\mathfrak{P}_{n,k}$ the set of those partitions with exactly 
$k$ classes, and by $\mathfrak{P}_n^{\ell}$ the set of those partitions 
with exactly $\ell$ singleton classes.  In addition, let $S(\mathcal{P})$ denote the set
$\{c:\{c\}\in\mathcal{P}\}$, so that 
$\mathfrak{P}_n^{\ell}=\{\mathcal{P}\in\mathfrak{P}_n:\left\vert S(\mathcal{P})\right\vert=\ell\}$.

A {\it semilabeled  binary tree} is tree that has only vertices of degree 3 and 1.
Degree one vertices are called {\it leaves}. A semilabeled  binary tree of {\it size $n$} 
has $n$ leaves, which are labeled bijectively 
with the elements of $[n]$. The remaining $n-2$ vertices are called {\it internal} vertices of the tree. 
We denote by $\LL(T)$ the set of leaves of $T$ (and we identify these leaves by their labels), 
as sometimes we use a label set different from $[n]$. 

{\it Caterpillars}  of  size $n$ are semilabeled  binary trees of  size $n$ such that the removal of
the $n$ leaves (and their incident edges) results in a path of length $n-3$ (here we assume that $n\ge 3$; 
for $n\in[2]$, we consider the unique $n$-leaf semilabeled tree as a caterpillar.) 
The {\it backbone} of the caterpillar, is the path obtained by removing all of its leaves.
If $\pi$ is a permutation on $[n]$, $T_{\pi}$ denotes the $n$-leaf caterpillar, where we draw the 
backbone as a straight line, the leaves appear on a straight line below the backbone, 
and the leaves are labeled in order as $\pi(1),\pi(2),\ldots,\pi(n)$.
Note that interchanging the first two or last two elements of a permutation results 
in the same labeled caterpillar, and so does completely reversing the order of entries. 
Consequently, there is only one labeled caterpillar for $n\le 3$ and there 
are $\frac{n!}{8}$  caterpillars for $n\ge 4$.
For the identity permutation $\id_n$ on $[n]$, we refer to $T_{\id_n}$ as the standard caterpillar. 
See Fig.~\ref{fig:standard} for some examples of these definitions.

\begin{figure}[http]
\center
\begin{tikzpicture}[scale=.5]
\node[vertexc] (a1) at (1,0) {$1$};
\node[vertexc] (a2) at (2,0) {$2$};
\node[vertexc] (a3) at (3,0) {$3$};
\node[vertexc] (a4) at (4,0) {$4$};
\node[vertexc] (a5) at (5,0) {$5$};
\node[vertexc] (a6) at (6,0) {$6$};
\node[vertexc] (a7) at (7,0) {$7$};
\node[vertex] (b2) at (2,1) {};
\node[vertex] (b3) at (3,1) {};
\node[vertex] (b4) at (4,1) {};
\node[vertex] (b5) at (5,1) {};
\node[vertex] (b6) at (6,1) {};
\draw[thick] (a1)--(b2)--(a2);
\draw[thick] (b2)--(b3)--(a3);
\draw[thick] (b3)--(b4)--(a4);
\draw[thick] (b4)--(b5)--(a5);
\draw[thick] (b5)--(b6)--(a6);
\draw[thick] (b6)--(a7);
\node at (4,-1) {$T_{\id_7}$};
\end{tikzpicture}
\qquad
\begin{tikzpicture}[scale=.5]
\node[vertexc] (a1) at (1,0) {$2$};
\node[vertexc] (a2) at (2,0) {$7$};
\node[vertexc] (a3) at (3,0) {$1$};
\node[vertexc] (a4) at (4,0) {$4$};
\node[vertexc] (a5) at (5,0) {$6$};
\node[vertexc] (a6) at (6,0) {$3$};
\node[vertexc] (a7) at (7,0) {$5$};
\node[vertex] (b2) at (2,1) {};
\node[vertex] (b3) at (3,1) {};
\node[vertex] (b4) at (4,1) {};
\node[vertex] (b5) at (5,1) {};
\node[vertex] (b6) at (6,1) {};
\draw[thick] (a1)--(b2)--(a2);
\draw[thick] (b2)--(b3)--(a3);
\draw[thick] (b3)--(b4)--(a4);
\draw[thick] (b4)--(b5)--(a5);
\draw[thick] (b5)--(b6)--(a6);
\draw[thick] (b6)--(a7);
\node at (4,-1) {$T_{(2,7,1,4,6,3,5)}$};
\end{tikzpicture}
\qquad
\begin{tikzpicture}[scale=.5]
\node[vertexc] (a1) at (1,0) {$7$};
\node[vertexc] (a2) at (2,0) {$2$};
\node[vertexc] (a3) at (3,0) {$1$};
\node[vertexc] (a4) at (4,0) {$4$};
\node[vertexc] (a5) at (5,0) {$6$};
\node[vertexc] (a6) at (6,0) {$5$};
\node[vertexc] (a7) at (7,0) {$3$};
\node[vertex] (b2) at (2,1) {};
\node[vertex] (b3) at (3,1) {};
\node[vertex] (b4) at (4,1) {};
\node[vertex] (b5) at (5,1) {};
\node[vertex] (b6) at (6,1) {};
\draw[thick] (a1)--(b2)--(a2);
\draw[thick] (b2)--(b3)--(a3);
\draw[thick] (b3)--(b4)--(a4);
\draw[thick] (b4)--(b5)--(a5);
\draw[thick] (b5)--(b6)--(a6);
\draw[thick] (b6)--(a7);
\node at (4,-1) {$T_{(7,2,1,4,6,5,3)}$};
\end{tikzpicture}
\caption{Some size $7$  caterpillars.}
\label{fig:standard}
\end{figure}
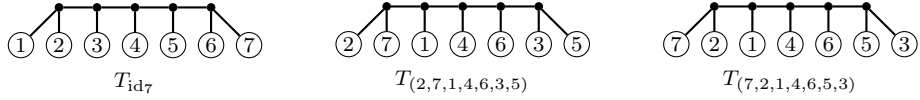

In a semilabeled  binary tree, the removal of $k-1\geq 0$ edges partitions the tree into $k$ connected components. Some of those,
however, may not contain leaves. So the removal of $k-1$ edges induces a partition of  $[n]$ into {\it at most} $k$ classes
by the relation ``being in the same connected component". Given a semilabeled  binary tree $T$, we denote by $\mathfrak{P}(T)$
the set of partitions of $[n]$ induced by the removal of some edges from the tree.
Note that, in the language used in the introduction, $\mathfrak{P}(T)$ 
corresponds to the set of convex characters on $T$;
and if $P \in \mathfrak{P}_n\setminus  \mathfrak{Q}_n$, then $P$ is a trivial character, and 
for every size $n$ semilabeled binary tree  $T$, $P \in \mathfrak{P}(T)$.

If $T_{\pi}$ is a caterpillar with $n$ leaves, 
and $\mathcal{P}\in\mathfrak{P}(T)\cap\mathfrak{Q}_{n,s}\cap\mathfrak{P}_{n,k}$ for some $0\le s\le k$, then $\mathcal{P}=\{P_1,\ldots,P_k\}$ is a \emph{standard listing of $\mathcal{P}$} if ($|P_i|\ge 2$ iff $1\le i\le s$). The listing
$\{P_1,\ldots,P_k\}$ is a \emph{standard listing of $\mathcal{P}$ with respect to $\pi$} if it is a standard listing and
for $1\le i <j\le s$ we have $\min\bigl(\pi^{-1}(P_i)\bigl)<\min\bigl(\pi^{-1}(P_j)\bigl)$. In particular, when $s=0$, then any listing of $\mathcal{P}$ is standard for any $\pi$; 
and for $s=1$,
a listing is standard for a permutation $\pi$ precisely when $P_1$ is the unique class of size at least $2$. 
Note that for any $t\in[n]$ we have that $t$ is in the $\pi^{-1}(t)$-th position
in $(\pi(1),\pi(2),\ldots,\pi(n)$.  We consider now  $s\ge 2$. For simplicity, we will assume $n\ge 4$, i.e., that
 $T_{\pi}$ has at least one backbone edge.
It is easy to see that
 if $E$ is an edge set of $T_{\pi}$ inducing $\mathcal{P}$, then $E$ contains none of the backbone edges
on the path from $\pi\left(\min(\pi^{-1}(P_i))\right)$ to $\pi\left(\max(\pi^{-1}(P_i))\right)$ in $T_{\pi}$, and for $i\ne s-1$ $E$ must contain at least one backbone edge
from the path from $\pi\left(\max(\pi^{-1}(P_i))\right)$ to $\pi\left(\min(\pi^{-1}(P_{i+1})\right)$. 
In fact, any edge set $E'$ of $T_{\pi}$ that contains precisely one backbone edge
from $\pi\left(\max(\pi^{-1}(P_i))\right)$ to $\pi\left(\min(\pi^{-1}(P_{i+1})\right)$ for each $1\le i\le s-1$ and contains all the leaf-edges leading to elements of $S(\mathcal{P})$ will generate
$\mathcal{P}$ in $T_{\pi}$.

The structure of $\mathfrak{P}(T)$ has been investigated in the literature.
Steel  \cite{jclass} showed that  
 $|\mathfrak{P}(T)\cap \mathfrak{P}_{n,k}|= \binom{2n-k-1}{k-1}$ and $\vert\mathfrak{P}(T)\vert$ is the Fibonacci number $F_{2n-1}$ for {\it every} semilabeled binary tree on $n$ leaves.
 Kelk and   Stamoulis \cite{convex} showed that 
$|\mathfrak{P}(T)\cap \mathfrak{P}_{n,k}\cap \mathfrak{Q}_n|= \binom{n-k-1}{k-1}$.

These results immediately imply that there is a $\frac{1}{2}<\alpha<1$ and $\epsilon>0$ 
such that for  $n\rightarrow\infty$ we have
\begin{equation} \label{uppertail}
\left\vert \mathfrak{P}(T)\cap \left( \cup_{\ell\geq \alpha n} \mathfrak{P}_n^\ell\right)\right\vert \leq \sum_{\ell\geq \alpha n}\binom{n}{\ell}\sum_k
\binom{n-\ell-k-1}{k-1}=O(2^{(1-\epsilon)n}),
\end{equation}
as all the  binomial coefficients are small when $\alpha$ is selected close to 1. We will need this estimate
in the proof of Theorem~\ref{fo}.

\section{Two trees}\label{sec:inequalities}

We begin by focusing on Question~\ref{question} for two trees. (In this special case we drop $t$ from the notation.) For $1\leq k\leq n$, define
\begin{eqnarray*}
s_{n,k}&=& \min |\mathfrak{P}(T)\cap \mathfrak{P}(F)\cap \mathfrak{P}_{n,k}| \text{ \ for semilabeled binary trees $T,F$ of size $n$,}\\
c_{n,k}&=& \min |\mathfrak{P}(T)\cap \mathfrak{P}(F)\cap \mathfrak{P}_{n,k}| \text{ \ for caterpillars $T,F$ of size $n$,}\\
s_{n}&=& \min |\mathfrak{P}(T)\cap \mathfrak{P}(F)| \text{ \ for semilabeled binary trees $T,F$ of size $n$,}\\
c_{n}&=& \min |\mathfrak{P}(T)\cap \mathfrak{P}(F)| \text{ \ for caterpillars $T,F$ of size $n$.}
\end{eqnarray*}

First, note that by definition
\begin{equation} \label{treeversuscat}
s_{n,k} \leq c_{n,k}  \text{\ and \ } s_n\leq c_n.
\end{equation}
We also have 
\begin{equation} \label{sum_k}
\sum_k s_{n,k} \leq  s_n   \text{\ and \ } \sum_k c_{n,k}\leq c_n,
\end{equation}
as the optimal tree (resp. caterpillar) pair providing $c_n$ (resp. $s_n$) is also considered when $s_{n,k}$ (resp. $c_{n,k}$) is defined.

\begin{lemma}\label{lm:monotone} For $n\ge 1$, we have $s_{n,1}=s_{n,n}=c_{n,1}=c_{n,n}=1$. Furthermore,
\begin{enumerate}[label=\upshape{(\roman*)}]
\item\label{monotone2var} for all $n\ge 3$ and $k\in[n-1]$, we have
 $s_{n,k} <  s_{n+1,k+1}$ and $c_{n,k} <  c_{n+1,k+1}$;
\item \label{monotone1var} for all $n\ge 1$,  we have $s_{n} <  s_{n+1}$ and $c_{n} <  c_{n+1}$;
\item\label{trivterm} for all $n\ge 2$ and $ k\in[n-1]$, we have $\binom{n}{k-1}\le s_{n,k}$;
\item\label{lowerbound_on_s_n} for all $n\ge 1$, we have $s_n\ge 2^n-n$.
\end{enumerate}
\end{lemma}
\begin{proof}
$1=s_{n,1}=c_{n,1}=s_{n,n}=c_{n,n}$ is obvious, as there is only one partition on $[n]$ that has $n$ partition classes and there is only one partition that has $1$ class, and these are clearly induced partitions for all semilabeled trees.

Let $T_1,T_2$ be trees with $n+1$ leaves for some $n\ge 1$.
Let $X\subseteq\mathfrak{P}(T_1)\cap\mathfrak{P}(T_2)\cap\mathfrak{P}_{n+1,k+1}$ be the set of partitions in which $\{n+1\}$ is a partition class.
For $i\in [2]$, create the $n$-leaf semilabeled trees $T_i'$ by removing the leaf labeled $n+1$ from $T_i$ and
 suppressing the degree 2 vertices created. 
Clearly, $\{\mathcal{P}\setminus \{\{n+1\}\}:\mathcal{P}\in X\}=\mathfrak{P}(T_1')\cap\mathfrak{P}(T_2')\cap\mathfrak{P}_{n,k}$, consequently
$|X|=\left\vert\mathfrak{P}(T_1')\cap\mathfrak{P}(T_2')\cap\mathfrak{P}_{n,k}\right\vert$. 
Moreover, 
if $n\ge 3$ and $k\in[n-1]$, then $\{\{i\}:i\in[k]\}\cup\{[n+1]-[k]\}\in (\mathfrak{P}(T_1)\cap\mathfrak{P}(T_2)\cap\mathfrak{P}_{n+1,k+1})\setminus X$,
in other words, $\left\vert\mathfrak{P}(T_1)\cap\mathfrak{P}(T_2)\cap\mathfrak{P}_{n+1,k+1}\right\vert>|X|$.
 Obviously when $T_1$ and $T_2$ are caterpillars, so are $T_1'$ and $T_2'$.

The strict inequalities in \ref{monotone2var} follow if we choose $T_1,T_2$ be optimal binary trees (resp. caterpillars) on $n+1$ leaves such that 
$|\mathfrak{P}(T_1)\cap\mathfrak{P}(T_2)\cap\mathfrak{P}_{n+1,k+1}|$ is $s_{n+1,k+1}$ (resp. $c_{n+1,k+1}$). 

If $T_1,T_2$ are optimal $(n+1)$-leaf trees (resp. caterpillars) providing the value $s_{n+1}$ (resp. $c_{n+1}$),  
 then
\begin{eqnarray*}
\left\vert\mathfrak{P}(T_1)\cap\mathfrak{P}(T_2)\right\vert&=&1+\sum_{k=1}^n\left\vert\mathfrak{P}(T_1)\cap\mathfrak{P}(T_2)\cap\mathfrak{P}_{n+1,k+1}\right\vert\\
&\ge&1+\sum_{k=1}^n\left\vert\mathfrak{P}(T_1')\cap\mathfrak{P}(T_2')\cap\mathfrak{P}_{n,k}\right\vert
>\left\vert\mathfrak{P}(T_1')\cap\mathfrak{P}(T_2')\right\vert,
\end{eqnarray*}
which implies~\ref{monotone1var}.

\ref{trivterm} holds since, for any two semilabeled binary trees, selecting any $k-1$ subset of leaf labels and deleting the corresponding
$k-1$ leaf edges results in the same size $k$ partition with from the two trees, but repeating this procedure for a different $k-1$ subset of leaves results
in a different partition as long as $k<n$. $s_1=1=2^1-1$, and for $n\ge 2$ the statement in \ref{lowerbound_on_s_n} follows 
from $s_{n,n}=1$, (\ref{sum_k}) and \ref{trivterm}.
\end{proof}

\begin{remark}
We are not aware of any examples where equality holds in (\ref{treeversuscat}), and it
would be interesting to show that these inequalities are always strict (see Section~\ref{sec:discuss}).
\end{remark}

\section{Lower bounds for $c_{n,k}$}\label{sec:lower-bounds}

In this section we prove the following lower bounds for $c_{n,k}$ and, as a corollary, also give 
a lower bound for $c_n$.

\begin{theorem} \label{symmet}
 Let $n$ be a multiple of $4$. If $n\ge 8$, then
for every $\ell$ with $\frac{n}{2} \leq \ell \leq n-4$,  
the partitions of any two caterpillars $T,F$ share some 
partitions consisting of exactly $2$ classes of size at least 2 and exactly $\ell$ singletons, in fact at least
\begin{equation}
\label{residuesum}
\sum_{i=\max(2,n-\ell-\frac{n}{4})}^{\min(n-\ell-2,\frac{n}{4})} \binom{n/4}{i}\binom{n/4}{n-\ell-i}
\end{equation}
many such partitions.
Consequently, for every $k$ with
$ \frac{3n}{4} \leq k \leq n-2$ we have
\begin{equation} \label{residuereduction}
c_{n,k}\geq {n\choose k-1}+\left(1-\frac{4}{n-k+3}\right) \binom{n/2 }{n-k+2}, 
\end{equation}
and for every $k$ with $\frac{n}{2}+2\le k\le \frac{3n}{4}-1$
\begin{equation} \label{smallresiduereduction}
c_{n,k}\geq {n\choose k-1}+\sum_{i=\frac{3n}{4}-k+2}^{\frac{n}{4}}\binom{n/4}{i}\binom{n/4}{n-k+2-i}.
\end{equation}
\end{theorem}

\begin{proof}  
Note that for $n\ge 8$ and $\ell\le n-4$ we have $2\le \min(n-\ell-2,\frac{n}{4})$, and for $n\ge 8$ and $\ell\ge\frac{n}{2}$ 
we have $n-\ell-\frac{n}{4}\le \min(n-\ell-2,\frac{n}{4})$, so the sum in (\ref{residuesum}) has a non-empty range.

Without loss of generality take $T=T_{\id_n}$ and $F=T_\pi$ for an arbitrary permutation $\pi$.
 Consider  $A=\{1,2,...,\frac{n}{2}\}$, $B=\{\frac{n}{2}+1,...,n\}$, $X=\pi(A)=\{\pi(1),\pi(2),...,\pi(\frac{n}{2})\}$, $Y=\pi(B)=\{\pi(\frac{n}{2}+1),...,\pi(n)\}$.
Observe $n=|A\cap X|+ |A\cap Y|+|B\cap X|+ |B\cap Y|$, and hence 
$$\frac{n}{2}\leq \max \left(|A\cap X|+|B\cap Y|, |A\cap Y|+|B\cap X|\right).$$
Assume that the first term achieves the maximum (the other case will follow similarly). Set $m=|A\cap X|$.
Then  $|B\cap Y|=|Y|-|A\cap Y|=|A|-|A\cap Y|=|A\cap X|=m$, and,  as $2m\ge\frac{n}{2}$, $m\ge\frac{n}{4}$.

Fix an $\ell$ such that $\frac{n}{2} \le\ell\le n-4$. 

Set $t=n-\ell\leq \frac{n}{2}$. Then $4\le t\le \frac{n}{2}\le 2m$. For a set $C$ and integer $j:0\le j\le |C|$, $\binom{C}{j}$ denotes the collection of  $j$-element subsets of $C$.
For every $i:\max(2, t-\frac{n}{4})\le i\le \min(t-2,\frac{n}{4})$, define the family of partitions
$$\mathbb{Z}_{i,t}=\left\{\{Z_1,Z_2\}\cup \left\{\{j\}:j\in[n]\setminus(Z_1\cup Z_2)\right\}: Z_1\in\binom{A\cap X}{i}, 
Z_2\in\binom{B\cap Y}{t-i}\right\}.$$
As per the conditions we have $2\le i\le\frac{n}{4}\le m$ and $2\le t-i\le\frac{n}{4}\le m$, $\mathbb{Z}_{i,t}$ is well defined.
As $t-2\ge 2$ and $t-\frac{n}{4}\le\frac{n}{4}$, for every $t$ there is at least one $i$ to be selected, so $\mathbb{Z}_{i,t}\ne\emptyset$. 
Clearly,
$\mathbb{Z}_{i,t}\subseteq\mathfrak{P}(T)\cap\mathfrak{P}(F)\cap\mathfrak{Q}_{n,2}\cap\mathfrak{P}_{n}^{\ell}$ and 
$|\mathbb{Z}_{i,t}|=\binom{m}{i}\binom{m}{t-i}\ge\binom{n/4}{i}\binom{n/4}{t-i}$. In particular,
\begin{eqnarray*}
\left\vert\mathfrak{P}(T)\cap\mathfrak{P}(F)\cap \mathfrak{Q}_{n,2}\cap\mathfrak{P}_n^{\ell}\right\vert\ge\sum_{i=\max(2,t-\frac{n}{4})}^{\min(t-2,\frac{n}{4})} \binom{n/4}{i}\binom{n/4}{t-i},
\end{eqnarray*}
This proves (\ref{residuesum}).

Assume now $\ell\geq \frac{3n}{4}-2$, and consequently $t=n-\ell\leq \frac{n}{4}+2$.
Now we have 
\begin{eqnarray*}
c_{n,\ell+2}&-&\binom{n}{\ell+1}\ge\sum_{i=\max(2,t-\frac{n}{4})}^{\min(t-2,\frac{n}{4})} \binom{n/4}{i}\binom{n/4}{t-i}
= \sum_{i=2}^{t-2} \binom{n/4}{i}\binom{n/4}{t-i}\\
&\geq& \binom{n/2}{t}\left( 1- \frac{4}{t+1}\right)= \binom{n/2}{n-\ell}\left( 1- \frac{4}{n-\ell+1}\right)  ,
\end{eqnarray*}
as the terms in this truncated Vandermonde convolution are symmetric and unimodal, and the missing 4 terms are at most 4 times the average expansion term.
This yields (\ref{residuereduction}).

On the other hand, if $\frac{n}{2}\le \ell <\frac{3n}{4}-2$, then for $t=n-\ell$ we have $\frac{n}{4}<t-2\le\frac{n}{2}$,
 $\min(t-2,\frac{n}{4})=\frac{n}{4}$ and $\max(2,t-\frac{n}{4})=t-\frac{n}{4}$, and as before, we have
\begin{eqnarray*}
c_{n,\ell+2}&-&\binom{n}{\ell+1}\ge
\sum_{i=t-\frac{n}{4}}^{\frac{n}{4}} \binom{n/4}{i}\binom{n/4}{t-i}
=\sum_{i=\frac{3n}{4}-\ell}^{\frac{n}{4}}\binom{n/4}{i}\binom{n/4}{n-\ell-i},
\end{eqnarray*}
which gives (\ref{smallresiduereduction}).
\end{proof}

We remark here that  $c_{n,n-1}=\binom{n}{n-3}$.

\begin{corollary}  \label{lowerbound_c_n}
We have  $c_n\geq 2^n+(\frac{1}{2}-o(1))2^{2\lfloor \frac{n}{4}\rfloor}.$ 
\end{corollary}
\begin{proof} 
Consider $n\ge 52$. This implies that $0.9n+3< n-2$, and for $k\le 0.9n+3$ we have $n-k+3\ge 0.1n$ and $1-\frac{4}{n-k+3}\ge 1-\frac{40}{n}$.  

First assume that  $n$ is a multiple of 4, and sum up the conclusion of Theorem~\ref{symmet}:
\begin{eqnarray*}
c_n-2^n+n&\ge&\sum_{k=1}^{n-2}\left(c_{n,k}-\binom{n}{k-1}\right)
\geq \sum_{k=\frac{3n}{4}}^{0.9n+3}\binom{n/2}{n-k+2}\left( 1- \frac{4}{n-k+3}\right)\\
&\geq& \left(1-\frac{40}{n}\right)\sum_{k=\frac{3n}{4}}^{0.9n+3}\binom{n/2}{n-k+2}
\geq\left(\frac{1}{2}-o(1)\right)2^{n/2},
\end{eqnarray*}
where in the last step we used the normal convergence of the binomial coefficients.

If $n$ is not a multiple of $4$, i.e. $n=4\lfloor \frac{n}{4}\rfloor +r $ with $r=1,2$ or $3$,
then on the one hand $\mathfrak{P}(T)\cap\mathfrak{P}(F)$ contains $2^n-n$ partitions with at most one class of size at least 2, while after
removing leaves $4\lfloor \frac{n}{4}\rfloor +1,...,4\lfloor \frac{n}{4}\rfloor +r$ from the trees $T$ and $F$ allows us 
to repeat the same argument,
using the second term from
(\ref{residuereduction}) with $4\lfloor \frac{n}{4}\rfloor$ substituted to the place of $n$, with a slightly changed range for $\ell$.
\end{proof}

\section{Caterpillars avoiding non-trivial common partitions}\label{sec:avoid}

In this section, we shall show that for small values of $k$, $c_{n,k}={n\choose k-1}$ actually holds.
To this end, we need to first make some observations concerning properties of partitions and sets of integers.

First, for partitions of sets of integers, we introduce the concept of a {\it gap}. For a finite set $X$ of positive integers, define
$\gap(X)=(\max\, X)- (\min\, X) +1-|X|$. Less formally, writing the elements of $X$ in increasing order, 
$\gap(X)$ is the number of integers that are missing
between consecutive elements of $X$.
For a family $\mathcal{X}$ of disjoint sets of positive integers (e.g., a partition of $[n]$), define $\gap(\mathcal{X})=\sum_{X\in \mathcal{X}} \gap(X)$.
Observe that 
\begin{equation} \label{positivity}
\gap(X)\geq 0\,  \text{\  and \ }\gap(\mathcal{X})\geq 0.
\end{equation}
Assume that every element of $X$ gives the same residue modulo $q$. Then it is easy to see that 
\begin{equation} \label{paritygap}
\gap(X)\geq (q-1)( |X|-1).
\end{equation} 

\begin{lemma}\label{lm:gap} The following hold
\begin{enumerate}[label=\upshape{(\roman*)}]
\item\label{idgap} If $\mathcal{P}\in\mathfrak{P}(T_{\id_n})$, then  $\gap(\mathcal{P})\le \left\vert S(\mathcal{P})\right\vert$.
\item If $\mathcal{P}\in\mathfrak{Q}_{n,0}$, then $\gap(\mathcal{P})=0$.
\item\label{standardgap} If $\mathcal{P}\in\mathfrak{Q}_{n,s}$ for some $s\ge 1$ and $\{P_1,\ldots,P_k\}$ is a standard listing of $\mathcal{P}$, 
then
$\gap(\mathcal{P})=\sum_{i=1}^s\gap(P_i)$.
\end{enumerate}
\end{lemma}

Now, for a set of positive integers $A$, $A_o$ and $A_e$ denotes the set of odd and even integers in $A$, respectively. We use the interval notation
$[a,b]=\{z\in\mathbb{Z}: a\le z\le b\}$. In other words, $[n]=[1,n]$ for positive integers $n$.
The following results will be frequently used in the proof of the next theorem. They
are straight-forward to show and  therefore presented without proof.
\begin{observation}\label{cl:cutting}
Let $1\le a<b$ be integers and $x\in[a,b]$. Then we have  
$$
\left\vert\left[a,x-1\right]_o\right\vert+\left\vert\left[x+1,b\right]_e\right\vert=
\begin{cases}
\left\vert\left[a,b\right]_e\right\vert-1, &a\equiv  0\mod 2 \\
\left\vert\left[a,b\right]_e\right\vert, &a\equiv1\mod 2\\
\left\vert\left[a,b\right]_o\right\vert-1, &b\equiv  1\mod 2 \\
\left\vert\left[a,b\right]_o\right\vert, &b\equiv0\mod 2
\end{cases}
$$
and
$$
\left\vert\left[a,x-1\right]_e\right\vert+\left\vert\left[x+1,b\right]_o\right\vert=
\begin{cases}
\left\vert\left[a,b\right]_o\right\vert-1, &a\equiv  1\mod 2 \\
\left\vert\left[a,b\right]_o\right\vert, &a\equiv0\mod 2\\
\left\vert\left[a,b\right]_e\right\vert-1, &b\equiv  0\mod 2 \\
\left\vert\left[a,b\right]_e\right\vert, &b\equiv1\mod 2.
\end{cases}
$$
Consequently, we have 
$$\min\Big(\left\vert\left[a,x-1\right]_o\right\vert+\left\vert\left[x+1,b\right]_e\right\vert,
\left\vert\left[a,x-1\right]_e\right\vert+\left\vert\left[x+1,b\right]_o\right\vert\Big)
\ge\left\lceil\frac{b-a-1}{2}\right\rceil.
$$
\end{observation}

Armed with the above observations, we now prove the main result of this section.

\begin{theorem} \label{megegyezik} For any $n\geq 1$ and 
$1\le k\le \lceil\frac{n}{3}\rceil$, we have 
 $c_{n,k}=\binom{n}{k-1}$. 
 \end{theorem}
 \begin{proof} 
The statement is true for $n\le 3$, as in this case  $\lceil\frac{n}{3}\rceil=1$, and $c_{n,1}=1=\binom{n}{0}$. 
  We assume for the rest of the proof that $n\ge 4$,
 i.e., $1\le \lfloor\frac{n}{3} \rfloor$ and $2\le\lceil\frac{n}{3}\rceil$. 
 
  Let $\pi$ be the permutation that lists the elements of the sets as given:
 first the elements of $\left[1,\lceil \frac{2n}{3}\rceil \right]_o$ in decreasing  order, 
 then  the elements of $\left[\lceil\frac{2n}{3}\rceil+1,n\right]_o$ in increasing order, 
 then the elements of $\left[\lceil\frac{2n}{3}\rceil+1,n\right]_e$ in decreasing order,
 then the elements of $[1,\lceil\frac{2n}{3}\rceil]_e$ in increasing order.  Note that $\pi$ lists the elements of 
 $[\lceil\frac{n}{3}\rceil+1,\lceil\frac{2n}{3}\rceil]_o$ before the elements of 
 $[1,\lceil \frac{n}{3}\rceil]_o$, and it lists the elements of
$[\lceil\frac{n}{3}\rceil+1,\lceil\frac{2n}{3}\rceil]_e$ after the elements of $[1,\lceil\frac{n}{3}\rceil]_e$.
Moreover,  $n=\lceil\frac{2n}{3}\rceil+\lfloor\frac{n}{3}\rfloor$ and 
$\lceil\frac{2n}{3}\rceil\ge\lceil\frac{n}{3}\rceil+\lfloor\frac{n}{3}\rfloor$.
In particular, none of the intervals
$[1,\lceil\frac{n}{3}\rceil]$, $[\lceil\frac{n}{3}\rceil+1,\lceil\frac{2n}{3}\rceil]$ and $[\lceil\frac{2n}{3}\rceil +1,n]$ is empty.

Assume to the contrary that for some $k\le\lceil \frac{n}{3}\rceil$ we have $\mathfrak{P}(T_{\id_n})\cap\mathfrak{P}(T_\pi)\cap\mathfrak{Q}_n\cap\mathfrak{P}_{n,k}\ne\emptyset$. 
Then for some $2\le s\le k\le\lceil\frac{n}{3}\rceil$ we have a 
$\mathcal{P}\in\mathfrak{P}(T_{\id_n})\cap\mathfrak{P}(T_{\pi})\cap\mathfrak{Q}_{n,s}\cap\mathfrak{P}_{n,k}$.
In particular, $|S(\mathcal{P})|=k-s$.

Let $\{P_1,\ldots,P_k\}$ be a standard listing of $\mathcal{P}$ with respect to $\pi$.
Since $\pi([n/2])=[n]_o$, there is a $t$ with $1\le t\le s$ such that for every $i$ with $1\le i<t$, $P_i$ contains only odd integers, 
and for every $i$ with $t<i\le s$,  $P_i$ contains only even integers. 
Moreover, as $\mathcal{P}\in\mathfrak{P}(T_{\id_n})$, we also
have that for every $i\in[s]\setminus\{t\}$ we have $P_i\subseteq [1,\min(P_t)-1]$ or $P_i\subseteq[\max(P_t)+1,n]$;
in particular $[\min(P_t),\max(P_t)]\setminus P_t\subseteq S(\mathcal{P})$.

We will bound $\gap(\mathcal{P})$.
By Lemma~\ref{lm:gap}~\ref{idgap}, $\gap(\mathcal{P})\le|S(\mathcal{P})|=k-s$. Using (\ref{paritygap}) with $q=2$, we have that
$\gap(P_i)\ge |P_i|-1$ for all $i\in[s]\setminus\{t\}$, therefore
\begin{eqnarray*}
\gap(\mathcal{P})&\ge& \sum_{i\in[s]\setminus\{t\}}\gap(P_i)\geq \left(\sum_{i\in[s]\setminus\{t\}}|P_i|\right)-(s-1)
=  \left\vert[n]\setminus\left(P_t\cup S(\mathcal{P})\right)\right\vert-s+1\\
&=& n-\left(|P_t|+k-s\right)-s+1
= n-k+1-|P_t|. 
\end{eqnarray*}
Comparing these two bounds gives
$$|P_t|\ge n-2k+s+1\ge n-2\left\lceil\frac{n}{3}\right\rceil+3\ge\left\lceil\frac{n}{3}\right\rceil+1.$$ 

If $P_t$ contains only odd or only even integers, then $|[\min(P_t),\max(P_t)]|\ge 2|P_t|-1$.
Since $[\min(P_t),\max(P_t)]\setminus P_t\subseteq S(\mathcal{P})$, we have
$k\ge \left\vert S(\mathcal{P})\right\vert+2 \ge (|P_t|-1)+2>\lceil\frac{n}{3}\rceil$,
a contradiction. Therefore $P_t$ must contain both even and odd integers, $\min(P_t)\le n-|P_t|+1\le \lceil\frac{2n}{3}\rceil$ and
$\max(P_t)\ge |P_t|>\lceil \frac{n}{3}\rceil $.
In addition, we also have 
$$\min(P_t)\le \left\lceil\frac{n}{3}\right\rceil \text{ or }  \max(P_t)\ge\left\lceil \frac{2n}{3}\right\rceil+1,$$
 otherwise we would have
$|P_t|\le\max(P_t)-\min(P_t)+1 \le \lceil \frac{2n}{3}\rceil -(\lceil\frac{n}{3}\rceil+1)+1\le\lceil \frac{n}{3}\rceil$, which is again a contradiction.

\noindent{\bf Case A:}
$\lceil\frac{n}{3}\rceil<\max(P_t)\le \lceil \frac{2n}{3}\rceil$.\\
 Then we have 
$P_t\cap [\lceil\frac{2n}{3}\rceil+1,n]=\emptyset$. As $P_t$ contains both even and odd integers and $\mathcal{P}\in\mathfrak{P}(T_{\pi})$, 
we have that
$[\lceil\frac{2n}{3}\rceil +1,n]\subseteq S(\mathcal{P})$, 
which in turn gives $k\ge |S(\mathcal{P})|+2\ge\lfloor\frac{n}{3}\rfloor +2>\lceil\frac{n}{3}\rceil$, a contradiction.
Therefore we must have $\max(P_t)\ge\lceil\frac{2n}{3}\rceil+1$.

\noindent{\bf Case B:}  $\lceil\frac{n}{3}\rceil+1\le\min(P_t)<\lceil\frac{2n}{3}\rceil$. \\
We have $[\lceil\frac{n}{3}\rceil+1,\lceil\frac{2n}{3}\rceil]\cap P_t\ne\emptyset$.
If neither $P_t\cap[\lceil\frac{n}{3}\rceil+1,\lceil\frac{2n}{3}\rceil]_o$ nor 
$P_t\cap[\lceil\frac{n}{3}\rceil+1,\lceil\frac{2n}{3}\rceil]_e$ is empty, then
from $\mathcal{P}\in\mathfrak{P}(T_{\pi})$ we get $[1,\lceil\frac{n}{3}\rceil]\subseteq S(\mathcal{P})$ and 
$k>\lceil\frac{n}{3}\rceil $,  a contradiction.
Assume that $P_t\cap[\lceil\frac{n}{3}\rceil+1,\lceil\frac{2n}{3}\rceil]_o\ne\emptyset$ and 
$P_t\cap[\lceil\frac{n}{3}\rceil+1,\lceil\frac{2n}{3}\rceil]_e=\emptyset$
(the other possibility follows similarly). As we have $[\min(P_t),\max(P_t)]\subseteq P_t\cup S(\mathcal{P})$, we have $[\min(P_t)+1,\lceil\frac{2n}{3}\rceil]_e\subseteq S(\mathcal{P})$.
From $\mathcal{P}\in \mathfrak{P}(T_{\pi})$ we have
$[1,\lceil\frac{n}{3}\rceil]_o\subseteq S(\mathcal{P})$ and 
$[\lceil\frac{n}{3}\rceil+1,\min(P_t)-1]_o\subseteq S(\mathcal{P})$.
These together with Claim~\ref{cl:cutting} give
\begin{eqnarray*}
|S(\mathcal{P})|&\ge& \left\vert\left[1,\left\lceil\frac{n}{3}\right\rceil \right]_o\right\vert
+\left\vert\left[\min(P_t)+1,\left\lceil\frac{2n}{3}\right\rceil\right]_e\right\vert
 +\left\vert\left[\left\lceil\frac{n}{3}\right\rceil+1, \min(P_t)-1\right]_o\right\vert\\ 
&\ge& \left\lceil\frac{1}{2}\left\lceil\frac{n}{3}\right\rceil\right\rceil
+\left\lceil\frac{1}{2}\left(\left\lceil\frac{2n}{3}\right\rceil
-\left\lceil\frac{n}{3}\right\rceil-2\right)\right\rceil
\ge \left\lceil\frac{1}{2}\left\lceil\frac{n}{3}\right\rceil\right\rceil
+\left\lceil\frac{1}{2}\left\lfloor\frac{n}{3}\right\rfloor\right\rceil-1\\
&\ge&\left\lceil\frac{n}{3}\right\rceil-1,
\end{eqnarray*} 
and $k\ge|S(\mathcal{P})|+2>\lceil\frac{n}{3}\rceil$, which is a contradiction. 
Thus we must have $\min(P_t)\le\lceil\frac{n}{3}\rceil$.

 \noindent{\bf Case C:} $\min(P_t)\le \lceil\frac{n}{3}\rceil$,  $\max(P_t)\ge\lceil\frac{2n}{3}\rceil+1$.\\
Recall that $P_t$ contains both even and odd integers. 
 Since for all $i\in[s]\setminus\{t\}$ we have
 $P_i\cap[\min(P_t),\max(P_t)]=\emptyset$, we have 
 $[\lceil\frac{n}{3}\rceil+1,\lceil\frac{2n}{3}\rceil]\setminus P_t\subseteq S(\mathcal{P})$. 
 
 \noindent{\bf Subcase a:} $ P_t\cap[\lceil\frac{n}{3}\rceil+1,\lceil\frac{2n}{3}\rceil]=\emptyset$.\\
 We have 
 $|S(\mathcal{P})|\ge|[\lceil\frac{n}{3}\rceil+1,\rceil\frac{2n}{3}\rceil]|=\lfloor\frac{n}{3}\rfloor\ge \lceil\frac{n}{3}\rceil-1$, 
 and $k\ge |S(\mathcal{P})|+2>\lceil\frac{n}{3}\rceil$, a contradiction. Thus we have
 $ P_t\cap[\lceil\frac{n}{3}\rceil+1,\lceil\frac{2n}{3}\rceil]\ne \emptyset$.
 
 \noindent{\bf Subcase b:}  $ P_t\cap[\lceil\frac{n}{3}\rceil+1,\lceil\frac{2n}{3}\rceil]$ contains both even and odd integers.\\
$\mathcal{P}\in\mathfrak{P}(T_{\rho})$ yields that for all $i\in[s]\setminus\{t\}$ we have
  $P_i\cap ([1,\lceil\frac{n}{3}\rceil]\cup[\lceil\frac{2n}{3}\rceil+1,n])=\emptyset$, which together with the earlier $P_i\cap[\lceil\frac{n}{3}\rceil+1,\lceil\frac{2n}{3}\rceil]=\emptyset$
  gives $s=1$, which is a contradiction. Therefore $ P_t\cap[\lceil\frac{n}{3}\rceil+1,\lceil\frac{2n}{3}\rceil]$ either
  contains only even integers, or only  odd integers.
  
  \noindent{\bf Subcase c:}  $P_t\cap[\lceil\frac{n}{3}\rceil+1,\lceil\frac{2n}{3}\rceil]_e=\emptyset$ and
$P_t\cap[\lceil\frac{n}{3}\rceil+1,\lceil\frac{2n}{3}\rceil]_o\ne\emptyset$.\\
 (Note that the other remaining case follows similarly, so we omit it.)\\
We have that for all $i\in[s]\setminus\{t\}$, $P_i\cap[1,\lceil\frac{n}{3}\rceil]_o=\emptyset$, and consequently
$[1,\lceil\frac{n}{3}\rceil]_o\setminus P_t\subseteq S(\mathcal{P})$ and also
$P_i\subseteq [1,\min(P_t)-1]_e$ or $P_i\subseteq[\max(P_t)+1,n]$.
If $[1,\lceil\frac{n}{3}\rceil]_e\cap P_t\ne\emptyset$, then from $\mathcal{P}\in\mathfrak{P}(T_{\rho})$ we get
for all $i\in[s]\setminus\{t\}$ that $P_i\cap [1,\lceil\frac{n}{3}\rceil]_e=\emptyset$ and $P_i\cap [\max(P_t)+1,n]=\emptyset$, in other words $s=1$,
a contradiction.
Therefore we have $[1,\lceil\frac{n}{3}\rceil]_e\cap P_t=\emptyset$.
 Since $P_t$ contains some even integers, $P_t\cap[\lceil\frac{2n}{3}\rceil+1,n]_e\ne\emptyset$.
In particular, we have that $[1,\min(P_t)-1]_o]\subseteq S(\mathcal{P})$ and $[\min(P_t)+1,\lceil\frac{2n}{3}\rceil]_e\subseteq S(\mathcal{P})$. 
Using Claim~\ref{cl:cutting} we get
\begin{eqnarray*}|\mathcal{S}(P)|\ge \left\vert\left[1,\min(P_t)-1\right]_o\right\vert
+\left\vert\left[\min(P_t)+1,\left\lceil\frac{2n}{3}\right\rceil\right]_e\right\vert
\ge\left\lceil\frac{1}{2}\left(\left\lceil\frac{2n}{3}\right\rceil-2\right)\right\rceil
\ge \left\lceil\frac{n}{3}\right\rceil-1,
\end{eqnarray*}
which gives $k\ge|S(\mathcal{P})|+2>\lceil\frac{n}{3}\rceil$, a contradicton.
 \end{proof}

\section{An upper bound for $c_n$}\label{sec:upper}

In this section, we prove that the following upper bound holds for $c_n$.

\begin{theorem} \label{fo}
For two uniformly and independently selected random permutations, create the corresponding caterpillars $F$ and $T$. Then we have 
\begin{equation}\label{elsofoformula}
\EE\left[\mathcal|\mathfrak{P}(T)\cap \mathfrak{P}(F)\cap \ \mathfrak{Q}_n|\right] = \Theta\left(\frac{2^n}{n^2}\right).  
\end{equation}
As a consequence, we have
\begin{equation}\label{masodikfoformula}
\EE\left[\mathcal|\mathfrak{P}(T)\cap \mathfrak{P}(F)|\right] = 2^n\left(1+\Theta\left(\frac{1}{n^2}\right)\right).  
\end{equation}
and
\begin{equation} \label{foformula}
c_n=2^{n}\left(1+O\left(\frac{1}{n^2}\right)\right).
\end{equation}
\end{theorem}
\begin{proof}
Clearly, $\EE\left[\mathcal|\mathfrak{P}(T)\cap \mathfrak{P}(F)\cap \ \mathfrak{Q}_n|\right] =\EE\left[\mathcal|\mathfrak{P}(T_{\id_n})\cap \mathfrak{P}(T_{\pi})\cap \ \mathfrak{Q}_n|\right] $, where the second expectation is computed over uniform selection of the permutation $\pi$.

In order to prove (\ref{foformula}),
by Lemma~\ref{lm:monotone}~\ref{lowerbound_on_s_n}, we only have to prove that the right hand side is an upper bound to $c_n$. 
Using (\ref{elsofoformula}) we get
$$
\EE\left[\mathcal|\mathfrak{P}(T)\cap \mathfrak{P}(F)|\right] = 2^n-n+\Theta\left(\frac{2^n}{n^2}\right)=2^n\left(1+\Theta\left(\frac{1}{n^2}\right)\right),  
$$
which is (\ref{masodikfoformula}). 
In turn, this gives that there is a
caterpillar $F'$, such that $|\mathfrak{P}(T_{\id_n})\cap \mathfrak{P}(F')|\leq 2^n\left(1+O(\frac{1}{n^2})\right)$,
 which implies (\ref{foformula}).

Therefore we only need to  bound   $\EE\left[\mathcal|\mathfrak{P}(T_{\id_n})\cap \mathfrak{P}(T_{\pi})\cap \ \mathfrak{Q}_n|\right]$. 
Clearly,
$$
\EE=\EE\left[\left\vert\mathfrak{P}(T_{\id_n})\cap \mathfrak{P}(T_{\pi})\cap \mathfrak{Q}_n\right\vert\right]
=\EE\left[\left\vert \mathfrak{P}(T_{\id_n})\cap \mathfrak{P}(T_{\pi})\cap \mathfrak{Q}_n\cap\left(\bigcup_{\ell=0}^n \mathfrak{P}_n^{\ell} \right)\right\vert\right]$$
By (\ref{uppertail}) there is $ \frac{1}{2}<\alpha<1$ and $\epsilon>0$ such that
\begin{eqnarray*}
\EE\left[\left\vert \mathfrak{P}(T_{\id_n})\cap \mathfrak{P}(T_{\pi})\cap \mathfrak{Q}_n\cap\left(\bigcup_{\ell\le \alpha n} \mathfrak{P}_n^{\ell} \right)\right\vert\right]
\le \left\vert\mathfrak{P}(T_{\id_n})\cap\left(\bigcup_{\ell\le \alpha n} \mathfrak{P}_n^{\ell} \right)\right\vert=O\left(2^{(1-\epsilon) n}\right).
\end{eqnarray*}
Therefore by linearity of expectation and the fact that $2^{(1-\epsilon)n}=o\left(\frac{2^n}{n^2}\right)$,
\begin{eqnarray*}
\EE&=&\left(\sum_{\ell=0}^{\alpha n} \EE\left[\left\vert \mathfrak{P}(T_{\id_n})\cap \mathfrak{P}(T_{\pi})\cap \mathfrak{Q}_n)\cap \mathfrak{P}_n^{\ell}   \right\vert \right]\right)+o\left(\frac{2^n}{n^2}\right)\\
&=& \left(\sum_{\ell=0}^{\alpha n}\sum_{\mathcal{P}\in \mathfrak{P}(T_{\id_n})\cap\mathfrak{Q}_n\cap \mathfrak{P}_n^{\ell}}  
 \mathbb{P}\left[ \mathcal{P}\in \mathfrak{P}(T_\pi) \right]\right)+o\left(\frac{2^n}{n^2}\right) \\
 &=&\left(\sum_{\ell=0}^{\alpha n}  \frac{1}{n!}  \sum_{\mathcal{P}\in \mathfrak{P}(T_{\id_n})\cap  \mathfrak{Q}_n\cap  \mathfrak{P}_n^{\ell} } 
 \left\vert\left\{\pi:  \mathcal{P}\in \mathfrak{P}(T_\pi)\right\} \right\vert\right)+o\left(\frac{2^n}{n^2}\right).
 \end{eqnarray*}
  
  We claim that for every $\ell$ and $m\geq 2$ we have
  \begin{equation}\label{ellksum}
\frac{1}{n!}  \sum_{\mathcal{P}\in \mathfrak{P}(T_{\id_n})\cap 
   \mathfrak{P}_n^{\ell} \cap \mathfrak{P}_{n,m+\ell}}\  \left\vert\left\{\pi:  \mathcal{P}\in \mathfrak{P}(T_\pi)\right\}     \right\vert = 
\binom{n}{\ell}\frac{m!}{(n-\ell)!} \sum_{{a_1+a_2+\ldots +a_m=n-\ell} \atop {a_1\geq 2,...,a_m\geq 2}} \prod_{j=1}^m (a_j!).
 \end{equation}

To verify this claim, take $\mathcal{P}\in \mathfrak{P}(T_{\id_n})\cap 
   \mathfrak{P}_n^{\ell} \cap \mathfrak{P}_{n,m+\ell}$.
   Then $\mathcal{P}$ has $\ell$ singleton classes and $m$ classes of size at least $2$. Let $\{P_1,P_2,\ldots, P_k\}$ be a standard listing of $\mathcal{P}$ with respect to $\id_n$, i.e., $|P_i|\ge 2$ precisely when $i\in[s]$, and for $1\le i<j\le s$ we have $\min(P_i)<\min(P_j)$. 
 For $i\in[s]$,  let $a_i=|P_i|$, then we have $a_i\ge 2$ and $a_1+\ldots+a_m=n-\ell$.
 It follows  that for $1\le i<j\le s$ we also have $\max(P_i)<\min(P_j)$, and listing the elements of $[n]-S(\mathcal{P})$ in increasing order, 
  for each $i<j$ all elements of $P_i$ come before any element of $P_j$. To create a permutation $\pi$ for which $\mathcal{P}\in T_{\pi}$, we need to put down the elements of $S(\mathcal{P})$ into the $n$ possible positions in any way 
  (which can be done in $\binom{n}{\ell}\ell!=\frac{n!}{(n-\ell!)}$ ways) 
  and then on the remaining $n-\ell$ places put any permutation of the remaining $n-\ell$ elements 
  that keep the elements of each $P_i$ consecutive 
  (which can be done in $m!\prod_i (a_i!)$ ways). 
 
 Therefore for a fixed $\mathcal{P}\in \mathfrak{P}(T_{\id_n})\cap 
   \mathfrak{P}_n^{\ell} \cap \mathfrak{P}_{n,m+\ell}$ with given $a_1,\ldots,a_m$ there are 
   $
   \frac{n!}{(n-\ell!)}m!\prod_i (a_i!)$ permutations $\pi$ such that $\mathcal{P}\in\mathfrak{P}(T_{\pi})$. We also need
   to count the partitions of $\mathfrak{P}(T_{\id_n})\cap 
   \mathfrak{P}_n^{\ell} \cap \mathfrak{P}_{n,m+\ell}$   that give the same ordered $m$-tuple 
   $(a_1,\ldots,a_m)$. However, that should be clear, as we can choose in $\binom{n}{\ell}$
   ways the elements that are to become singleton classes, and the rest of the permutation is given by the fixed $(a_1,\ldots,a_m)$
 Formula (\ref{ellksum}) follows. This implies that
 \begin{equation}\label{ellsum}
\frac{1}{n!} \sum_{\mathcal{P}\in \mathfrak{P}(T_{\id_n})\cap 
 \mathfrak{Q}_n\cap  \mathfrak{P}_n^{\ell} }\left\vert\left\{\pi:  \mathcal{P}\in \mathfrak{P}(T_{\pi})\right\}     \right\vert =
 \binom{n}{\ell}\sum_{m=2}^{\lfloor\frac{n-\ell}{2}\rfloor} \frac{m!}{(n-\ell)!} \sum_{{a_1+a_2+\ldots +a_m=n-\ell} \atop {a_1\geq 2,...,a_m\geq 2}} \prod_{j=1}^m (a_j!).
 \end{equation}
 
   The proof of the theorem boils down to giving a good upper bound on the sum of the right hand side of (\ref{ellsum}) over the range
  $ \ell \leq \alpha n$. 
  More precisely, give good uniform upper bounds for the sum following the binomial coefficient $\binom{n}{\ell}$  in the range above.
  
  To do so, we set
  \begin{equation}\label{cmellclean}
 C_{m}^{\ell}= \frac{m!}{(n-\ell)!} \sum_{{a_1+a_2+\ldots +a_m=n-\ell} \atop {a_1\geq 2,...,a_m\geq 2}} \prod_{j=1}^m (a_j!).
    \end{equation}

 For each $m$, consider the tuples $(a_1,\ldots,a_m)$ in which  exactly $s$ of the $a_i$  equals to 2. 
 If $s=m$, then we have $n-\ell$ is even, $m=\frac{n-\ell}{2}$, and $\prod_{j=1}^m, (a_j!)=2^m=2^{\frac{n-\ell}{2}}$, and we have
 $C_{(n-\ell)/2}^{\ell}=\frac{2^{\frac{n-\ell}{2}}}{(n-\ell)_{(n-\ell)/2}}\le\left(\frac{4}{n-\ell}\right)^{\frac{n-\ell}{2}}$. 
   Assume $s<m$. 
We can select the $s$ indices for which $a_i=2$ in $\binom{m}{s}$ ways, moreover, if $a_j\ne 2$ we have $a_j\ge 3$.
 The number of tuples with exactly $s$ of the $a_i$ equals $2$ (and consequently the remaining $m-s$ terms are each at least $3$ and
 add up to
 $n-\ell-2s$) is
 $$\binom{m}{s}\binom{(n-\ell-2s)-3(m-s)+(m-s)-1}{(m-s)-1}=\binom{m}{s}\frac{(n-\ell-2m-1)!}{(m-s-1)!(n-\ell-3m+s)!}.$$ 
   As for any $4\le r\le q$ we have $r!q!<(r-1)!(q+1)!$ and $r+q=(r-1)+(q+1)$, if for each $(a_1,\ldots,a_m)$ where exactly $s$ of the $a_i$ equals to $2$ we have
   that 
   $$\prod_{i=1}^m(a_i!)\le 2^s (3!)^{m-s-1}(n-\ell-2s-3(m-s-1))!=2^s6^{m-s-1}(n-\ell-3m+s+3)!$$
  and we must have $n-\ell-3m+s+3\ge 3$, in other words $\max(0,3m-(n-\ell))\le s\le m-1$.

  Therefore,when $m\le \lfloor\frac{n-\ell-1}{2}\rfloor$ we have
  \begin{eqnarray}
  C_m^{\ell}&\le&\frac{m!}{(n-\ell)!}\sum_{s=\max(0,3m-(n-\ell))}^{m-1}
  \frac{2^s6^{m-s-1}\binom{m}{s}(n-\ell-2m-1)!}{(m-s-1)!(n-\ell-3m+s)!}(n-\ell-3m+s+3)!\nonumber\\
  &=&\sum_{s=\max(0,3m-(n-\ell))}^{m-1}
  \frac{2^s6^{m-s-1}\binom{m}{s}m!(n-\ell-2m-1)!(n-\ell-3m+s+3)!}{(m-s-1)!(n-\ell)!(n-\ell-3m+s)!}\nonumber\\
  &=&\sum_{s=\max(0,3m-(n-\ell))}^{m-1}
  \frac{2^s6^{m-s-1}\binom{m}{s}m_{s+1}(n-\ell-3m+s+3)_3}{(n-\ell)_{2m+1}}\label{cmell}
    \end{eqnarray}
  
 We want to bound the following quantity by breaking it up to $3$ sums (where we keep in mind that $\ell\le \alpha n$, and $n\rightarrow\infty$)
 \begin{eqnarray} 
\sum_{\ell=0}^{\alpha n}\sum_{m=2}^{\lfloor\frac{n-\ell}{2}\rfloor}C_m^{\ell}&=&
\sum_{\ell=0}^{\alpha n}\sum_{m=2}^4 C_m^{\ell} 
  +\sum_{\ell=0}^{\alpha n}\sum_{m=5}^{\lfloor\frac{n-\ell}{3}\rfloor} C_m^{\ell}
+ \sum_{\ell=0}^{\alpha n}\sum_{m=\lfloor\frac{n-\ell}{3}\rfloor+1}^{\lfloor\frac{n-\ell}{2}\rfloor} C_m^{\ell}. \label{tobound}
  \end{eqnarray}
   
 {\bf Case:} $ \frac{n-\ell}{3}<m\leq \lfloor \frac{n-\ell}{2}\rfloor$.\\
 We have  $(m)_{s+1}\leq \lfloor \frac{n-\ell}{2}\rfloor!$, $(n-\ell-3m+s+3)_3\leq n^3$,
  $\sum_s 2^s6^{m-s} \binom{m}{s} =8^m<8^\frac{n-\ell}{2}$.
 For $\lfloor \frac{n-\ell}{3}\rfloor<m\leq \lfloor \frac{n-\ell-1}{2}\rfloor$ formula (\ref{cmell}) gives
\begin{eqnarray*}
C_m^{\ell}&\le& \frac{n^3 \lfloor \frac{n-\ell}{2}\rfloor! 8^{\frac{n-\ell}{2}}}{(n-\ell)_{2\lfloor\frac{n-\ell}{3}\rfloor+1}}
\le\frac{n^3 8^{\frac{n-\ell}{2}}}{(\lceil\frac{n-\ell}{2}\rceil)_{\lfloor\frac{n-\ell}{6}\rfloor+1}}\cdot\frac{\lfloor\frac{n-\ell}{2}\rfloor! }{(n-\ell)_{\lfloor\frac{n-\ell}{2}\rfloor}}
<\frac{n^3 8^{\frac{n-\ell}{2}}}{(\lceil\frac{n-\ell}{2}\rceil)_{\lfloor\frac{n-\ell}{6}\rfloor+1}}\\
&<&\frac{n^3 8^{\frac{n-\ell}{2}}}{\left(\frac{n-\ell}{3}\right)^{\frac{n-\ell}{6}}}<n^3\left(\frac{3\cdot 8^3}{n-\ell}\right)^{\frac{n-\ell}{6}}
<n^3\left(\frac{3\cdot 8^3}{(1-\alpha)n}\right)^{\frac{n}{6}}
\end{eqnarray*}
and if $\frac{n-\ell}{2}$ is an integer, then 
$$C_{(n-\ell)/2}^{\ell}\le\left(\frac{4}{n-\ell}\right)^{\frac{n-\ell}{2}}<\frac{n^3 8^{\frac{n-\ell}{2}}}{\left(\frac{n-\ell}{3}\right)^{\left\lfloor\frac{n-\ell}{6}\right\rfloor}}
<n^3\left(\frac{3\cdot 8^3}{(1-\alpha)n}\right)^{\frac{n}{6}}
$$
and
$$
\sum_{m=\lfloor\frac{n-\ell}{3}\rfloor+1}^{\lfloor\frac{n-\ell}{2}\rfloor}C_m^{\ell}<n^4\left(\frac{3\cdot 8^3}{(1-\alpha)n}\right)^{\frac{n}{6}}.
$$
As the above quantity goes to $0$ superexponentially fast as $n\rightarrow\infty$, we get
\begin{equation}\label{partone}
\sum_{\ell=0}^{\alpha n}\binom{n}{\ell}\sum_{m=\lfloor\frac{n-\ell}{3}\rfloor+1}^{\lfloor\frac{n-\ell}{2}\rfloor}C_m^{\ell}
<2^nn^4\left(\frac{3\cdot 8^3}{(1-\alpha )n}\right)^{\frac{n}{6}}=o\left(\frac{2^n}{n^2}\right).
\end{equation}
  
  {\bf Case:} $5\leq m\leq \frac{n-\ell}{3}$.\\
For $\ell\le \alpha n$ we have $n-\ell\ge (1-\alpha )n>7$ for large enough $n$, therefore  for large enough $n$ for $0\le \ell\le \alpha n$ we have 
$\frac{n-\ell}{3}<\frac{n-\ell}{2}$, and consequently $s<m$.
As when $n\rightarrow \infty$ and $\ell\le \alpha n$ we have $(n-\ell)\rightarrow\infty$, for large enough $n$ there is a constant $K$ such that
 $$\frac{(n-\ell-3m+s+3)_3}{(n-\ell)_{2m+1}}
 =\frac{(n-\ell-3m+s+3)_3}{(n-\ell-2m+2)_3}\cdot\frac{1}{(n-\ell)_{2m-2}}<\frac{K}{(n-\ell)_{2m-2}}.$$
  As before, $\sum_s 2^s6^{m-s} \binom{m}{s} =8^m$, and $s\le m-1$, $m\le\frac{n-\ell}{3}$ together with (\ref{cmell}) gives
\begin{eqnarray*}
  C_m^{\ell}&\le& \frac{K\cdot 8^m\left(\lfloor\frac{n-\ell}{3}\rfloor\right)_m}{(n-\ell)_{2m-2}}
 =\frac{K\cdot 8^m}{(n-\ell)_{m-2}}\cdot\frac{\left(\lfloor\frac{n-\ell}{3}\rfloor\right)_m}{(n-\ell-m+2)_{m}}
  <\frac{K\cdot 8^m}{(n-\ell)_{m-2}}<\frac{K\cdot 8^m}{\left(\frac{n-\ell}{3}\right)^{m-2}}\\
 & =&64K\left( \frac{24}{n-\ell}\right)^{m-2}<64K\left(\frac{24}{(1-\alpha)n}\right)^{m-2}.
    \end{eqnarray*}
  Now we have for large enough $n$
  \begin{eqnarray*}
  \sum_{m=5}^{\lfloor\frac{n-\ell}{3}\rfloor} C_m^{\ell}
  &\leq&
  \sum_{5\leq m\leq \frac{n-\ell}{3}}    64K\left( \frac{24}{(1-\alpha)n}\right)^{m-2}
  <64 K \sum_{k=3}^{\infty}\left( \frac{24}{(1-\alpha)n}\right)^{k}\\
  &=&\frac{64\cdot 24^3K}{(1-\alpha)^3n^3}\frac{1}{1-\frac{24}{(1-\alpha)n}}
  <\frac{70\cdot 24^3K}{(1-\alpha)^3n^3}
  \end{eqnarray*}
  and consequently
  \begin{equation}\label{parttwo}
  \sum_{\ell=0}^{\alpha n}\binom{n}{\ell}\sum_{m=5}^{\lfloor\frac{n-\ell}{3}\rfloor}C_m^{\ell}<2^n\cdot\frac{70\cdot 24^3K}{(1-\alpha)^3n^3}
  =o\left(\frac{2^n}{n^2}\right).
  \end{equation}

 { \bf Case:} $2\leq m\leq 4$\\
 Again, $s\le m-1$, $m\le 4$ and as $n\rightarrow\infty$, $n-\ell\rightarrow\infty$  there is a constant $D$ such that
 \begin{eqnarray*}
 \frac{(m)_{s+1}(n-\ell-3m+s+3)_3}{(n-\ell)_{2m+1}}
 \le \frac{m!(n-\ell-2m+2)^3}{(n-\ell-2m)^{2m+1}}=\frac{D}{(n-\ell-2m)^{2m-2}}
 \end{eqnarray*}
When $2\le m\le 4$, by (\ref{cmell}) we get
$C_{m,\ell}=O\left(\frac{1}{n^{2m-2}}\right)= O\left(\frac{1}{n^2}\right)$, and consequently
$$\sum_{\ell=0}^{\alpha n}\sum_{m=2}^4 C_m^{\ell}=O\left(\frac{2^n}{n}\right).$$

Now (\ref{cmellclean}) gives that 
\begin{eqnarray*}
C_2^{\ell}= 2 \sum_{a_1=2}^{n-\ell-2} \frac{1}{\binom{n-\ell}{a_1}}\ge 2\frac{1}{\binom{n-\ell}{2}}>\frac{8}{(n-\ell)^2}\ge\frac{8}{n^2}.
\end{eqnarray*}
Now we get by any large deviation theorem and $\alpha>\frac{1}{2}$ that for large enough $n$
$$
  \sum_{\ell=0}^{\alpha n}\binom{n}{\ell}\sum_{m=2}^{4}C_m^{\ell}
  \geq\frac{8}{n^2}\sum_{\ell=0}^{\alpha n}\binom{n}{\ell}
  \ge\frac{8}{n^2}\left(1+o(1)\right)2^n=\Omega\left(\frac{2^n}{n^2}\right),
  $$
  which gives
  \begin{equation}\label{partthree}
  \sum_{\ell=0}^{\alpha n}\binom{n}{\ell}\sum_{m=2}^{4}C_m^{\ell}=\Theta\left(\frac{2^n}{n^2}\right) . 
  \end{equation}

Now (\ref{elsofoformula}) follows from (\ref{ellsum}), (\ref{cmellclean}), (\ref{tobound}), (\ref{partone}), (\ref{parttwo}) and (\ref{partthree}).
\end{proof}
 
 \section{More than two trees}\label{sec:more}

In this section, we give some bounds for the number of coconvex characters 
on a collection of $n$-leaved trees using results from the theory
of maximum agreement subtrees. Analogues of the quantities $s_{n,k}, c_{n,k},s_{n}, c_{n}$
can clearly be defined for minimizing over $t\geq 3$ semilabeled binary trees or caterpillars, and
as in the introduction
these analogous quantities will be denoted by $s_{n,k}^{(t)}, c_{n,k}^{(t)},s_{n}^{(t)}, c_{n}^{(t)}$.

\subsection{Agreement subtrees} 

We first state some of the results from the theory of agreement forests.
An {\it  induced binary subtree }    $T|_Y$   of a semilabeled tree $T$ induced by a subset of leaves $Y$ 
 is defined as follows:
take the subtree induced by $Y$ in $T$, and
substitute paths in which all internal vertices have degree 2
 by edges. $T|_Y$ is a semilabeled binary tree, but its labels are from $Y$, not from $[n]$.

An important algorithmic problem, known as the {\em Maximum Agreement Subtree
Problem}, is the following: given two semilabeled binary trees, $T,F$, find a
common induced binary subtree of the largest possible size. In other words, find a maximum size set $Y\subseteq[n]$ such that $T|_y\simeq F|_Y$.
Somewhat surprisingly, this
problem can be solved in polynomial time \cite{stewar} (see also \cite{GKKMcM} and \cite{KKMcM95}).
The maximum agreement subtree problem is closely related to the topic of this paper, and in part motivated it.
Let $\mast(n)$ denote  the smallest order (number of leaves, or vertices)
 of the maximum
agreement subtree of two  semilabled trees with the same $n$ leaves.
In 1992, Kubicka, Kubicki, and McMorris \cite{KKMcM} showed that
$d_1 (\log \log n)^{1/2} < \mast(n) < d_2 \log n$ with some
explicit constants. After a long series of improvements, Markin \cite{markin} settled that $\mast(n) =\Theta(\log n)$.

Similarly, one can define the size of the  maximum agreement subtree of $t$ trees, $F_1,F_2,...,F_t$, denoted by
$\mast(F_1,F_2,...,F_t)$, and extend the extremal problem as 
\begin{equation} \label{agreementmany}
\mast_t(n)=\min_n \mast(F_1,F_2,...,F_t),
\end{equation}
where $F_1,F_2,...,F_t$ are binary semilabeled trees on the same $n$ leaves. 
One defines $\mastcat_t(n)$ as in (\ref{agreementmany}), but the minimum is taken over $t$-tuples of
caterpillar trees.

For caterpillar trees, there is always a substantially bigger agreement subtree. The reason is that a permutation always contains 
a $\lfloor\sqrt{n-1}\rfloor $ length increasing or decreasing subsequence, due to the Erd\H os--Szekeres theorem on sequences.
Assume without loss of generality that one of our caterpillar trees is $T_{\id_n}$, described by the identity permutation, and the other by
an arbitrary permutation $\pi$. Leaves labeled by the    $\lfloor\sqrt{n-1}\rfloor $ length increasing or decreasing subsequence
give an agreement subtree.

We need asymptotic lower bounds for $\mast_t(n)$ and $\mastcat_t(n)$. 
Here $t$ is fixed and $n\rightarrow \infty$.
The following bounds can be shown by induction from the cited result of Markin \cite{markin}
and the observation above:
\begin{eqnarray}
c \underbrace{\log\log \cdots \log }_{t-1 {\rm\ times}} n &\leq & \mast_t(n)   \label{alt} \hbox{\ \ for some\ } c>0\hbox{\ and \ }\\ 
n^{2^{1-t}}-t+1&\leq  &\mastcat_t(n).\label{forcat}
\end{eqnarray}

\subsection{Application of agreement subtrees}
 We show that with fixed number $t$ of caterpillars or trees, for sufficiently large $n$ and $k$, the trees still have a common partition
 in $\mathfrak{Q}_n$. Formally,
 \begin{theorem} 
 For every fixed  $t\ge 2$, for every sufficently large $n$, for all $k$ with $n-g_t(n)\leq  k\leq n-4$,
 $$s^{(t)}_{n,k}>\binom{n}{k-1};$$
 and for every $t$, for every sufficently large $n$, for all $k$ with $n-f_t(n)\leq k\leq n-4$,
 $$c^{(t)}_{n,k}>\binom{n}{k-1}.$$
 The choice $f_t(n)= n^{2^{1-t}}-t+1$
  and $g_t(n)=c' \underbrace{\log\log \cdots \log }_{t-1 {\rm\ times}} n$ with a certain $c'>0$  suffices.
  \end{theorem}  
 \begin{proof}
 Any $g_t(n)\leq \max\{s: \mast_{t}(n)\geq 2s\}$ and $f_t(n)\leq \max\{s: \mastcat_{t}(n)\geq 2s\}$ functions suffice. Indeed,  
 with the bounds for $k$ we have $4\le n-k\le g_t(n)$ (resp. $4\le n-k\le f_t(n)$). 
 Therefore with any $t$ given trees  (resp. caterpillars), we can select a size $2(n-k)$ common subtree 
 (resp. subcaterpillar) of them. From the leaves of that we can select $n-k$
 two-element common partition classes for the given $t$ trees (caterpillars) of size $n$, while keeping the remaining $2k-n$ leaves all singletons
 (note that the bounds imply that for sufficiently large $n$ we have $2k-n>0$).
 This yields a partition of $n$ elements into $n-k$ (at least $2$) doubletons and $2k-n$ singletons, i.e., a partition into $k$ classes. In view of (\ref{alt})
 and (\ref{forcat}), the limitation for $k$ is sufficient.
 \end{proof}
 \begin{theorem} \label{morethantwo}
 For every   fixed $0<\delta<\frac{1}{3}$, and $n $ sufficiently large, there are at most $\left\lceil\frac{3}{\delta}\right\rceil-2$ caterpillar trees with $n$ leaves,
 such that for all $k<(1-\delta)n$, there is no $k$-partition $\mathcal{P}\in \mathfrak{Q}_n$ that is a common partition of these caterpillars.
 In other words, with the conditions above,
 $$c_{n,k}^{(3\lceil 1/\delta\rceil -2)}=\binom{n}{k-1}.$$
\end{theorem}
\begin{proof}
Fix  $m=\lceil 1/\delta \rceil$ and observe for large enough $n$ we have $3\leq m<n$. 
For $i:0\le i\leq m-1$ let the block $D_i$ be the set of integers in $[n]$ that are $i$ modulo $m$.
Let $\rho$ be an arbitrary permutation that lists the element of $[n]$ in the order $D_0,D_1\ldots,D_{m-1}$ (where the order within each block is arbitrary).
Using $\rho$, we define a few more permutations:
\begin{enumerate}[label=(\roman*)]
\item  $\rho_i$ for $i=0,1...,m-1$ is obtained from $\rho$ by reversing the order among the elements of $D_i$ and not changing anything else ($m$ new permutations). 
\item $\rho^{(0,i)}$ for $i=1...,m-2$ is obtained from $\rho$ by interchanging the blocks $D_0$ and $D_i$  and not changing anything else ($m-2$ new permutations). 
\item $\rho^{(i,m-1)}$ for $i=1...,m-2$ is obtained from $\rho$ by interchanging the blocks $D_i$ and $D_{m-1}$  and not changing anything else ($m-2$ new permutations). 
\end{enumerate}
Create caterpillar trees $T_{\pi}$ for  permutations $\pi$ of the form $\id_n$, $\rho$, $\rho_i$, $\rho^{(0,i)}$ and $\rho^{(i,m-1)}$, a total of $3m-2$ trees.
\begin{claim} \label{egybe}
If $\mathcal{P}\in \mathfrak{Q}_n$ is a common partition of the $3m-2$ trees above, then for every $P\in \mathcal{P}$
there is an $i$ such that $P\subseteq D_i$.
\end{claim}
\begin{proof} The claim is trivially true for singletons, so we need to consider only partition classes of size at least $2$. Let $P\in\mathcal{P}$ such that
$|P|\ge 2$, and let $P'\in\mathcal{P}\setminus\{P\}$ such that $|P'|\ge 2$. Since $\mathcal{P}\in\mathfrak{Q}_n$, $P'$ exists.
Set $H=\{j:P\cap D_j\ne\emptyset\}$ and $H'=\{j:P'\cap D_j\ne\emptyset\}$. 
Then $H,H'\ne\emptyset$. 
Assume to the contrary that $|H|\ge 2$, i.e., $\min(H)<\max(H)$; in particular $1\le \max(H)$ and $\min(H)\le m-2$.

Choose an $i$ such that  $P'\cap D_i\ne\emptyset$ (i.e. $i\in H'$).
Since $\mathcal{P}\in\mathfrak{P}(T_{\rho})$, we must have $i\le \min(H)$ or $i\ge \max(H)$.
If $i\in\{\min(H),\max(H)\}$, then both $P\cap D_i$ and $P'\cap D_i$ are non-empty, therefore
$\mathcal{P}\notin\mathfrak{P}(T_{\rho_i})$ (this uses the fact that $|H|\ge 2$), which is a contradiction.
Therefore $i<\min(H)$ or $i>\max(H)$; in particular we have $0<\min(H)$ or $\max(H)<m-1$.
We also must have $\max(H')<\min(H)$ or $\max(H)<\min(H')$: this is trivially true when $|H'|=1$, otherwise we can just exchange the role of $P$ and $P'$ in the previous argument.
If $0\le \max(H')<\min(H)\leq m-2$,
then  $\mathcal{P}\notin\mathfrak{P}(T_{\rho^{(0,\min(H))}})$, a contradiction.
If  $1\leq \max(H)< \min(H')\le m-1$, then $\mathcal{P}\notin\mathfrak{P}(T_{\rho^{(\max(H),m-1)}})$, another contradiction.
Therefore $P$ cannot intersect more than one block, and the claim follows.
\end{proof}

To conclude the proof of the theorem, consider a common partition $\mathcal{P}\in \mathfrak{Q}_n\cap\mathfrak{P}_{n,k}$ of the $3m-2 $ caterpillars, and let $s$ be the number
of size at least $2$ partition classes of $\mathcal{P}$. Let $\{P_1,\ldots,P_k\}$ be a standard listing of $\mathcal{P}$.
By Claim~\ref{egybe} there is a sequence $0\le q_1<q_2<\cdots<q_{s}\le m-1$ such that for every $j\in [s]$ we have $P_j\subseteq D_{q_j}$. By (\ref{paritygap}) we have for every $i\in[s]$, $\gap(P_i)\ge(m-1)(|P_i|-1)$, consequently
$\gap(\mathcal{P})\ge (m-1)(n-| S(\mathcal{P})|-s)$.  As $|S(\mathcal{P})|=k-s$, we get $\gap(\mathcal{P})\ge (m-1)(n-k)$.
On the other hand, Lemma~\ref{lm:gap}~\ref{idgap} gives $\gap(\mathcal{P})\leq |S(\mathcal{P})|= k-s$. Comparison of these two results 
gives $k-s\geq (m-1)(n-k)$, or equivalently 
$$k\geq \frac{m-1}{m}n+\frac{s}{m}>\left(1-\frac{1}{m}\right)n\ge(1-\delta)n,$$
contradicting the assumption $k<(1-\delta)n$.
\end{proof}

\section{A new family of tree metrics}\label{sec:metrics}

We now present the application of coconvex characters mentioned in the introduction.
Let $ n\ge 4$ and $2\le k\le n-2$. We define the {\em $k$-character distance $d_k(T,F)$} between two $n$-leaf 
trees $T,F$ to be 
$$d_k(T,F)=\left\vert \left(\mathfrak{P}(T)\Delta\mathfrak{P}(F)\right)\cap\mathfrak{P}_{n,k}\right\vert,$$
and the {\em character distance $d(T,F)$} to be
$$d(T,F)=\left\vert \mathfrak{P}(T)\Delta\mathfrak{P}(F)\right\vert.$$
We first show that $d_k$ and $d$ are both tree metrics.

\begin{lemma} Let $n\ge 4$. Then for all $2\le k\le n-2$, $d_k$ is a 
metric on the $n$-leaf trees. Moreover,
$$
d(T,F)=\sum_{k=2}^{n-2}d_k(T,F),
$$
and so $d$ is also a metric on the collection of $n$-leaf trees.
\end{lemma}
\begin{proof}
As $\mathfrak{P}(T)\cap\mathfrak{P}_{n,1}=[n]$ and  $\mathfrak{P}(T)\cap(\mathfrak{P}_{n,n-1}\cup\mathfrak{P}_{n,n})
\subseteq\mathfrak{Q}_{n,0}\cup\mathfrak{Q}_{n,1}$, it is clear that $d(T,F)=\sum_{k=2}^{n-2}d_k(T,F)$. So we only need to
prove the statement concerning $d_k$.

Clearly, $d_k(T,F)=d_k(F,T)\ge 0$, $d_k(T,T)=0$, and, as
$|A\Delta B|+|B\Delta C|\ge |A\Delta C|$, $d_k(\cdot,\cdot)$ satisfies the triangle inequality. Thus 
to show that $d_k$ is a metric it suffices to show that for all $n$-leaved trees $T\ne F$ we have $d_k(T,F)>0$.

So, suppose that $T,F$ are two trees with $n$ leaves such that $T\ne F$. Then, in particular, $n\ge 4$, and so there
are $4$ leaves $a,b,c,d$ such that the quartet subtrees spanned by these leaves in $T$ and $F$ are different. Assume
without loss of generality that the bipartition corresponding to the central edge 
of the quartet in $T$ is $ab|cd$ while in $F$ it is $ac|bd$. 
Let $k$ be an integer satisfying $2\le k\le n-2$. Then $0\le k-2\le n-4$, and consequently we can select a  $k-2$-element subset $A$
of $[n]\setminus\{a,b,c,d\}$. Let $e_0$ be one of the edges in $T$ that separates $a,b$ from $c,d$, and let
$e_1,\ldots,e_{k-2}$ be the leaf-edges of $T$ incident to the elements in $A$. 
Then removing $\{e_i:i\in[k-2]\cup\{0\}\}$ from $T$
generates a partition $\mathcal{P}\in\mathfrak{Q}_{n,2}$, where 
$\mathcal{P}=\{C_1,C_2\}\cup\{\{a\}:a\in A\}$ and $a,b\in C_1$, $c,d\in C_2$. 

Now, consider the tree $F$ and an edge $f$ of $F$. 
The removal of $f$ either does not separate $a,b,c,d$, or separates exactly one of $a,b,c,d$ from the other three, or
or it puts $a,c$ in one class and  $b,d$ in another class. 
Thus, $\mathcal{P}\notin\mathfrak{P}(F)$, showing that $d_k(T,F)>0$.
\end{proof}

We now show that $d_2$ and $d_{n-2}$ are well known tree metrics.
Recall that the Robinson-Foulds distance between two $n$-leaf trees is the number of bipartitions
contained one but not both of the trees \cite{Robinson1981}, and that the  
quartet distance 
is the number of 4-tuples of leaves that span different quartet trees in the two trees \cite{Steel1993}.
We will show that $d_2$ is the Robinson-Foulds distance 
and $d_{n-2}$ is the quartet distance on $n$-leaved trees.
To this end, we shall employ the following useful observation.

\begin{lemma}\label{lm:basedistance}  Let $T$ be an $n$-leaf tree, and $\mathcal{P}\in\mathfrak{P}(T)$. Then 
\begin{enumerate}[label=\upshape{(\roman*)}]
\item\label{case:notoneless}  $|S(\mathcal{P})|\ne n-1$.
\item\label{case:onlytriv}  If $|S(\mathcal{P})|\in\{n,n-2,n-3\}$, then $\mathcal{P}\in\mathfrak{Q}_{n,0}\cup\mathfrak{Q}_{n,1}$. 
\item\label{case:single}  
If $|S(\mathcal{P})|\ge n-3$, then $|S(\mathcal{P})|\ge 2|\mathcal{P}|-n$.
\end{enumerate}
\end{lemma}

\begin{proof}
Let $j=|S(\mathcal{P})|$, and $k=|\mathcal{P}|$. We have that $j\le k\le n$.

Since the partition classes of $\mathcal{P}$ are disjoint and cover all elements of $[n]$, $j\ne n-1$, which is \ref{case:notoneless}. If $j\ge n-3$, then $\mathcal{P}$ can have only one class with size bigger than $1$ (as two such classes would cover at least $4$ elements), which implies~\ref{case:onlytriv}.
Moreover,
 $$n=\left\vert\bigcup_{P\in\mathcal{P}}P\right\vert\ge j+2(k-j)=2k-j,$$
i.e. $j\ge 2k-j$, which is~\ref{case:single}.
\end{proof}

\begin{prop}
$d_2$ is the Robinson-Foulds distance on $n$-leaf trees and for $n\ge 4$, $d_{n-2}$ is the quartet distance.
\end{prop}
\begin{proof}
Let $T$ be an $n$-leaf tree.
Since elements in $\mathfrak{P}(T)\cap \mathfrak{P}_{n,2}$ 
are bipartitions, $d_2$ is trivially the Robinson-Foulds distance.

Assume that $n\ge 4$.
By Lemma~\ref{lm:basedistance}, for any 
$\mathcal{P}\in\mathfrak{P}(T)\cap\mathfrak{P}_{n,n-2}\cap\mathfrak{Q}_n$ we have  $|S(\mathcal{P})|= n-4$.
This implies that
$\mathcal{P}$ has exactly $2$ partition classes of size exactly $2$ and $n-4$ partition classes
of size $1$. Thus, $\mathcal{P}$ can be obtained from $T$ by removing the $n-4$ leaf-edges leading to the leaves corresponding 
to elements of $S(\mathcal{P})$, and removing an extra edge separating the remaining $4$-leaves into
two classes of size $2$, i.e., these two classes form a bipartition induced
by the central edge of a quartet tree. Hence $d_{n-2}$ is the quartet distance.
\end{proof}

We conclude this section by explaining how the distance $d_k$ related to the quantity $s_{n,k}$.
As we cited before,  $|\mathfrak{P}(T)\cap \mathfrak{P}_{n,k}|= \binom{2n-k-1}{k-1}$ by \cite{SteelFibonacci}. Thus 
$$d_k(T,F)=2\binom{2n-k-1}{k-1}-2\left\vert\mathfrak{P}(T)\cap\mathfrak{P}(F)\cap\mathfrak{P}_{n,k}\right\vert$$
and
$$d(T,F)=2\sum_{k=2}^{n-2}\binom{2n-k-1}{k-1}-2\left\vert\mathfrak{P}(T)\cap\mathfrak{P}(F)\right\vert. $$
Therefore determining $s_{n,k}$ is equivalent to finding the largest possible 
value of $d_k$ over all possible pairs of $n$-leaved trees, that is, the {\em diameter} of $d_k$.
As mentioned in the introduction, even though the diameter of 
the Robinson-Foulds distance is well-known to be $2n-6$ \cite{Steel1993}, 
there is a 40 year-old conjecture by Bandelt and Dress that concerns 
the asymptotics of the diameter of the quartet distance.
Note also that various algorithms have been devised to compute the quartet distance (see e.g. \cite{nielsen2011sub}); it would be
interesting to see if similar algorithms can be found for computing $d_k$.

\section{Open problems}\label{sec:discuss}

Naturally defined combinatorial sequences are very frequently unimodal, consisting of an increasing run followed by a decreasing run, see \cite{Stanley1989}.
 Unimodality is often caused by log-concavity, i.e., the concavity of the sequence of the logarithms of the terms.
\begin{enumerate}
\item For two fixed $n$-leaf phylogenetic trees, $T$ and $F$, is the sequence of the number of common partitions,   $|\mathfrak{P}(T)\cap \mathfrak{P}(F)\cap \mathfrak{P}_{n,k}|$, unimodal 
for $1\leq k\leq n$? Is the sequence log-concave?
\item What is the answer of the previous question for  the sequence  of the number of \emph{non-trivial} common partitions,     $|\mathfrak{P}(T)\cap \mathfrak{P}(F)\cap \mathfrak{P}_{n,k}\cap \mathfrak{Q}_n|$?
\item What is the answer for the analogue of the previous two questions for $t>2$ fixed $n$-leaf phylogenetic trees?
\item For a fixed $n$ and $1\leq k \leq n$, are the sequences $s_{n,k}$ and  $c_{n,k}$ unimodal? Log-concave?
\item For a fixed $n$ and $1\leq k \leq n-1$, are the sequences $s_{n,k}-\binom{n}{k-1}$, $c_{n,k}-\binom{n}{k-1}$ unimodal? Log-concave?
\item Assuming that for a fixed $n$ the $c_{n,k}$ sequence is unimodal, which is first $k$ with $c_{n,k}>\binom{n}{k-1}$? (In Theorems~\ref{megegyezik} and \ref{symmet} we have shown $n/3<k< n/2+2$.)
\item Show separation between $s_{n,k}$ and $c_{n,k}$ or show that $s_{n,k}=c_{n,k}$. (the above cited conjecture of Bandelt and Dress \cite{bandeltdress} and the results of Alon et al. \cite{alon} exhibit the difficulty
of separating  $s_{n,n-2}$ from $c_{n,n-2}$.)
\item Show separation between  $s_n$ and $c_n$ or show that they are equal.
\item Find the expected number of common $k$-partitions for two uniformly selected random semilabelled binary trees of size $n$.
\item Is it possible to strengthen Theorem~\ref{morethantwo}? Is there a fixed $t$ such that $s_{n,k}^{(t)} =\binom{n}{k-1}$ for some large $k$? (say for all $k<n-f(n)$ for some $f(n)=o(n) $ function.)
\item How fast we can compute the $d_k$ distance of $n$-leaved phylogenetic trees? (Distance $d_{n-2}$ can be computed in $O(n \log n) $ time \cite{dudek,nielsen2011sub}.)
\end{enumerate}

\section*{Acknowledgements}
The authors SK and VM thank the Institut Mittag-Leffler for inviting them 
to the ``Emerging Mathematical Frontiers in Molecular Evolution" conference, where they
began to discuss coconvexity.
 Part of this work was done while EC, VM and LS were in residence at the Institute for Computational and Experimental Research in Mathematics (ICERM) in Providence (RI, USA) during the \emph{Theory, Methods, and Applications of Quantitative Phylogenomics} program (supported by grant DMS-1929284 of the National Science Foundation (NSF)).

\bibliographystyle{siam}
\bibliography{main}

\end{document}